\documentclass{amsart}

\usepackage{a4wide,latexsym,amssymb,amsthm,amsmath,color}

\theoremstyle{plain}

\newtheorem{theorem}{Theorem}
\newtheorem{lemma}{Lemma}
\newtheorem{proposition}{Proposition}

\theoremstyle{remark}

\newtheorem{remark}{Remark}

\def\R{\mathbb{R}}
\def\C{\mathbb{C}}

\newcommand{\nks}{\ensuremath{\mathbb{S}^3 \times
\mathbb{S}^3}}

\begin{document}

\title[Lagrangian submanifolds of the nearly K\"ahler  $\mathbb{S}^3 \times \mathbb{S}^3$]{Lagrangian submanifolds of the nearly K\"ahler  $\mathbb{S}^3 \times \mathbb{S}^3$ from minimal surfaces in  $\mathbb{S}^3$}

\author[B. Bekta\c s]{Burcu Bekta\c s}
\address{Istanbul Technical University, Faculty of Science and Letters, Department of Mathematics, 34469, Maslak, 
Istanbul, Turkey} 
\email{bektasbu@itu.edu.tr}

\author[M. Moruz]{Marilena Moruz}
\address{LAMAV, ISTV2 Universit\'e de Valenciennes, Campus du Mont Houy, 59313 Valenciennes Cedex 9,  France}
\email{marilena.moruz@gmail.com}

\author[J. Van der Veken]{Joeri Van der Veken}
\address{KU Leuven, Department of Mathematics, Celestijnenlaan 200B -- Box 2400, BE-3001 Leuven, Belgium} 
\email{joeri.vanderveken@wis.kuleuven.be}

\author[L. Vrancken]{Luc Vrancken}
\address{LAMAV, ISTV2 \\Universit\'e de Valenciennes\\ Campus du Mont Houy\\ 59313 Valenciennes Cedex 9\\ France and KU  Leuven \\ Department of Mathematics \\Celestijnenlaan 200B -- Box 2400 \\ BE-3001 Leuven \\ Belgium}  
\email{luc.vrancken@univ-valenciennes.fr}

\thanks{This work was partially supported by the project 3E160361 (Lagrangian and calibrated submanifolds) of the KU Leuven research fund and part of it was carried out while the first author visited KU Leuven, supported by The Scientific and Technological Research Council of Turkey (TUBITAK), under grant 1059B141500244.}

\begin{abstract}
We study non-totally geodesic Lagrangian submanifolds of the nearly K\"ahler $\nks$ for which the projection on the first component is nowhere of maximal rank. We show that this property can be expressed in terms of the so called angle functions and that such Lagrangian submanifolds are closely related to  minimal surfaces in $\mathbb{S}^3$. Indeed, starting from an arbitrary minimal surface, we can construct locally a large family of such Lagrangian immersions, including one exceptional example.
\end{abstract}
\date{}
\maketitle

\section{Introduction}
 The nearly K\"ahler manifolds are almost Hermitian manifolds with almost complex structure $J$ for which the tensor field $\tilde \nabla J$ is skew-symmetric, where $\tilde\nabla$ is the Levi Civita connection. They  have been studied intensively in the 1970's by Gray (\cite{9}). Nagy (\cite{nagy}, \cite{nagy2}) made further contribution to the classification of nearly K\"ahler manifolds and more recently it has been shown by Butruille (\cite{6}) that the only homogeneous 6-dimensional nearly K\"ahler manifolds are the nearly K\"ahler 6-sphere $\mathbb{S}^6$, $ \mathbb{S}^3\times  \mathbb{S}^3$,  the projective space $\mathbb{C}P^3$ and the flag manifold $SU(3)/U(1)\times U(1)$, where the last three are not endowed with the standard metric. All these spaces are compact 3-symmetric spaces. Note that, recently, the first complete non homogeneous nearly K\"ahler structures were discovered on $\mathbb{S}^6$ and $\nks$  in \cite{Haskins}.\\
A natural question for the above mentioned four homogeneous nearly K\"ahler manifolds is to study their submanifolds. There are two natural types of submanifolds of nearly K\"ahler (or more generally, almost Hermitian) manifolds, namely almost complex and totally real submanifolds. Almost complex submanifolds are submanifolds whose tangent spaces are invariant under $J$. For a totally real submanifold, a tangent vector is mapped by the almost complex structure $J$ into a normal vector. In this case, if additionally the dimension of the submanifold is half the dimension of the ambient manifold, then the submanifold is Lagrangian.\\
Note that the Lagrangian submanifolds of nearly K\"ahler manifolds are especially interesting as they are always minimal and orientable (see \cite{Ejiri2} for $\mathbb{S}^6$ or \cite{Schafer}, \cite{Ivanov} for the general case).
Lagrangian submanifolds of $\mathbb{S}^6$ have been studied by many authors (see, amongst others, \cite{Dillen-Vrancken-Verstr}, \cite{Dillen-Vrancken}, \cite{Ejiri1}, \cite{Ejiri2},\cite{Vrancken1}, \cite{Vrancken}, \cite{Lotay} and \cite{Palmer}), whereas the study of Lagrangian submanifolds of $\nks$ only started recently. The first examples of those were given in \cite{Schafer} and \cite{Moroianu}.
Moreover, in \cite{diooshuetal} and \cite{dioosvranckenwang}, the authors obtained a classification of the Lagrangian submanifolds, which  are either totally geodesic or have constant sectional curvature. An important tool in the study in \cite{dioosvranckenwang} and \cite{diooshuetal}  is the use of an almost product structure $P$ on $\nks$, which was introduced in \cite{complexsurfaces}. Its definition is recalled in Section \ref{prelim}. The decomposition of $P$ into a tangential part and a normal part along a Lagrangian submanifold allows us to introduce three principal directions, $E_1,E_2,E_3$, with corresponding angle functions $\theta_1,\theta_2,\theta_3$.\\
In this paper we are interested in studying non-totally geodesic Lagrangian submanifolds $f:M\rightarrow \nks : x\mapsto f(x)=(p(x),q(x))$, for which the first component has nowhere maximal rank. Basic properties of the structure and its Lagrangian submanifolds are given in Section \ref{prelim}. In Section \ref{Results}, we show that in this case $\theta_1=\frac{\pi}{3}$ (Theorem \ref{th1}) and $p(M)$ has to be a (branched) minimal surface in $\mathbb{S}^3$ (Theorem \ref{th3}).
Conversely, for a non-totally geodesic minimal surface in $\mathbb{S}^3$ which, corresponds to a solution of the Sinh-Gordon equation, $\Delta\omega=-8\sinh\omega$, and for an additional arbitray solution of the Liouville equation, $\Delta\mu=-e^\mu$, we can construct locally a Lagrangian immersion of $\nks$. Thus, we obtain a large class of examples of Lagrangian immersions. 
We also obtain that a similar class of Lagrangian immersions can be associated to a totally geodesic surface in $\mathbb{S}^3$. This last case contains, in particular, the constant curvature sphere obtained in \cite{dioosvranckenwang}. Additionally, for each non-totally geodesic minimal surface, we obtain also one  exceptional example. In case of the Clifford torus in $\mathbb{S}^3$, this additional example is the flat Lagrangian torus in $\nks$ discovered in \cite{dioosvranckenwang}. We also show that any non-totally geodesic Lagrangian immersion for which the first component  has nowhere maximal rank is obtained by applying one of the three previously mentioned constructions. The main results are summarized in Section \ref{conclusion}.

\section{Preliminaries}\label{prelim}
In this section we recall the homogeneous nearly K\"ahler structure of  $\mathbb{S}^3\times\mathbb{S}^3$ and we mention some of the known results from \cite{dioosvranckenwang} and \cite{diooshuetal}.\\
By the natural identification $T_{(p,q)}(\mathbb{S}^3\times\mathbb{S}^3)\cong T_p\mathbb{S}^3\oplus T_q\mathbb{S}^3$, we write a tangent vector at $(p,q)$  as $Z(p,q)=(U(p,q),V(p,q))$ or simply $Z=(U,V)$. We regard the $3$-sphere as the set of all unit quaternions in $\mathbb{H}$ and we use the notations $i,j,k$ to denote the imaginary units of $\mathbb{H}$. In computations it is often  useful to write a tangent vector $Z(p,q)$ at $(p,q)$ on $\mathbb{S}^3\times\mathbb{S}^3$ as $(p\alpha,q\beta)$, with $\alpha$ and $\beta$  imaginary quaternions. This is possible as for $v\in T_p\mathbb{S}^3$ we know that $\langle v,p\rangle=0$ and, in addition,  for $p\in \mathbb{S}^3$ we can always find $\tilde v\in \mathbb{H}$ such that $v=p\tilde v$ . Moreover, $Re(\tilde v)=0$ as it follows from $\langle v,p\rangle=0$ when we expand $p$ and $\tilde v$ in the basis $\{1,i,j,k\}$.\\
We define the vector fields 
\begin{equation} \label{baza}
\begin{array}{ll}
\tilde E_1(p,q)=(pi,0),&\tilde  F_1(p,q)=(0,qi),\\
\tilde E_2(p,q)=(pj,0),& \tilde  F_2(p,q)=(0,qj),\\
\tilde E_3(p,q)=-(pk,0),&\tilde   F_3(p,q)=-(0,qk),
 \end{array} \end{equation}
which are mutually orthogonal with respect to the usual Euclidean product metric  on  $\mathbb{S}^3\times\mathbb{S}^3$. The Lie brackets are $[\tilde E_i,\tilde E_j]=-2\varepsilon_{ijk}\tilde E_k,$ $[\tilde F_i,\tilde F_j]=-2\varepsilon_{ijk}\tilde F_k$ and $[\tilde E_i,\tilde F_j]=0$, where
\begin{align*}
\varepsilon_{ijk}=\left\{
\begin{array}{l}
1, \quad \text{if} \ (ijk)\ \text{is an even permutation of } (123),\\
-1, \quad \text{if} \ (ijk)\ \text{is an odd permutation of } (123),\\
0, \quad \text{otherwise.}
\end{array}
\right.
\end{align*}
The almost complex structure $J$ on the nearly K\"ahler $\mathbb{S}^3\times\mathbb{S}^3$ is defined by 
\begin{equation}
J(U,V)_{(p,q)}=\frac{1}{\sqrt{3}} (2pq^{-1}V-U, -2qp^{-1}U+V ),
\end{equation}
for $(U,V)\in T_{(p,q)}(\mathbb{S}^3\times\mathbb{S}^3)$. 
The nearly K\"ahler metric on $\mathbb{S}^3\times\mathbb{S}^3$ is the Hermitian metric associated to the usual Euclidean product metric on $\mathbb{S}^3\times\mathbb{S}^3$: 
\begin{align}\label{g}
g(Z,Z')=&\frac{1}{2}(\langle Z,Z' \rangle+ \langle JZ,JZ' \rangle)\\
	=&\frac{4}{3}(\langle U,U' \rangle+\langle V,V'\rangle)-\frac{2}{3}(\langle p^{-1}U,q^{-1}V' \rangle+\langle p^{-1}U',q^{-1}V \rangle),\nonumber
\end{align}
where $Z = (U, V )$ and $Z' = (U', V')$. In the first line $\langle \cdot,\cdot \rangle $ stands for the usual
Euclidean product metric on $\mathbb{S}^3\times\mathbb{S}^3$ and in the second line $\langle \cdot,\cdot \rangle $ stands for the usual Euclidean metric on $\mathbb{S}^3$. By definition, the almost complex structure is compatible with the metric $g$.\\
From \cite{complexsurfaces} we have the following lemma.
\begin{lemma} 
The Levi-Civita connection $\tilde\nabla$ on $\mathbb{S}^3\times \mathbb{S}^3$ with respect to the metric $g$ is given by 
\[ \begin{array}{ll}
\tilde\nabla_{\tilde E_i}\tilde E_j=-\varepsilon_{ijk}\tilde E_k & \tilde\nabla_{\tilde E_i}\tilde F_j=\frac{\varepsilon_{ijk}}{3}(\tilde E_k-\tilde F_k) \\
\tilde\nabla_{\tilde F_i}\tilde E_j=\frac{\varepsilon_{ijk}}{3}(\tilde F_k -\tilde E_k)& \tilde\nabla_{\tilde F_i}\tilde F_j=-\varepsilon_{ijk}\tilde F_k. 
 \end{array} \]
\end{lemma}
Then we have that 
\begin{equation}\label{nablaJ}
 \begin{array}{ll}
(\tilde\nabla_{\tilde E_i}J)\tilde E_j=-\frac{2}{3\sqrt{3}}\varepsilon_{ijk}(\tilde E_k+2\tilde F_k), &    (\tilde\nabla_{\tilde E_i}J)\tilde F_j=-\frac{2}{3\sqrt{3}}\varepsilon_{ijk} (\tilde E_k-\tilde F_k),\\
(\tilde\nabla_{\tilde F_i}J)\tilde  E_j=-\frac{2}{3\sqrt{3}} \varepsilon_{ijk} (\tilde E_k-\tilde F_k), & (\tilde\nabla_{\tilde F_i}J)\tilde F_j=-\frac{2}{3\sqrt{3}}\varepsilon_{ijk}(2\tilde E_k+\tilde F_k)
 \end{array} 
\end{equation}

Let  $G:=\tilde\nabla J$. Then $G$ is skew-symmetric and  satisfies that 
\begin{equation}
G(X,JY)=-JG(X,Y), \quad g(G(X,Y),Z)+g(G(X,Z),Y)=0,
\end{equation}
for any vectors fields $X,Y,Z$ tangent to $\mathbb{S}^3\times\mathbb{S}^3$.
Therefore, $\mathbb{S}^3\times\mathbb{S}^3$ equipped with $g$ and $J$, becomes a nearly K\"ahler manifold.\\
The almost product structure $P$ introduced in \cite{complexsurfaces} is defined as 
\begin{equation}\label{defP}
PZ=(pq^{-1}V,qp^{-1}U), \quad\forall Z=(U,V)\in T_{(p,q)}(\mathbb{S}^3\times\mathbb{S}^3) 
\end{equation}
plays an important role in the study of the Lagrangian submanifolds of the nearly K\"ahler $\mathbb{S}^3\times\mathbb{S}^3$. It has the following properties:
\begin{align*}
&P^2=Id \quad \text{($P$ is involutive)},\\
&PJ=-JP \quad \text{($P$ and $J$ anti-commute)},\\
&g(PZ,PZ')=g(Z,Z') \quad \text{($P$ is compatible with $g$)},\\
&g(PZ,Z')=g(Z,PZ') \quad \text{($P$ is symmetric)}.
\end{align*}
Moreover, the almost product structure $P$ can be expressed in terms of the usual product structure $QZ=Q(U,V)=(-U,V)$ and vice versa:
\begin{align*}
QZ &= \frac{1}{\sqrt{3}} (2 PJZ - JZ),\\
PZ &= \frac{1}{2}(Z-\sqrt{3}QJZ). 
\end{align*}
Next, we recall the relation between the Levi-Civita connections $\tilde\nabla$ of $g$ and $\nabla^E$ of the Euclidean product metric $\langle\cdot,\cdot\rangle$.
\begin{lemma}\cite{dioosvranckenwang}
The relation between the nearly K\"ahler connection $\tilde\nabla$ and the Euclidean connection $\nabla^E$is
$$
\nabla^E_XY=\tilde\nabla_XY+\frac{1}{2}(JG(X,PY)+JG(Y,PX)).
$$ 
\end{lemma}
We recall here a useful formula, already known in  \cite{dioosvranckenwang}.\\
Let $D$ be the Euclidean connection on $\R^8$. For vector fields $X=(X_1,X_2)$ and $Y=(Y_1,Y_2)$ on $\nks$, we may decompose $D_XY$ along the tangent and the normal directions as follows: 
\begin{equation}\label{label}
D_XY=\nabla^E_XY+\frac{1}{2}\langle D_XY,(p,q) \rangle (p,q)+\frac{1}{2}\langle D_XY,(-p,q) \rangle (-p,q).
\end{equation}
Here, notice the factor $\frac{1}{2}$ due to the fact that  $(p,q)$ and $(-p,q)$ have  length $\sqrt{2}$. Moreover, as $\langle Y,(p,q)\rangle=0$, \eqref{label} is equivalent with
$$
D_XY=\nabla^E_XY-\frac{1}{2}\langle Y,X \rangle (p,q)-\frac{1}{2}\langle Y,(-X_1,X_2) \rangle (-p,q).
$$
In the special case that $Y_2=0$, the previous formula reduces to 
\begin{equation}\label{labe}
D_X(Y_1,0)=\nabla^E_X(Y_1,0)-\langle X_1, Y_1 \rangle(p,0).
\end{equation}\\
We find it appropriate here to prove an additional important  formula not explicitly mentioned in \cite{complexsurfaces}, that allows us to evaluate $G$  for any tangent vector fields.
\begin{proposition}
 Let $X=(p\alpha,q\beta),Y=(p\gamma,q\delta)\in T_{(p,q)}\nks$. Then 
\begin{align}\label{formulaG}
G(X,Y)=\frac{2}{3\sqrt{3}}(p(\beta\times\gamma+\alpha\times\delta&+\alpha\times\gamma-2\beta\times\delta),
	q(-\alpha\times\delta-\beta\times\gamma+2\alpha\times\gamma-\beta\times\delta)).
\end{align}
\end{proposition}
\begin{proof}
As $\alpha$ is an imaginary unit quaternion, we may write $\alpha=\alpha_1\cdot i+\alpha_2\cdot j+\alpha_3\cdot k$ and similarly for $\beta,\gamma,\delta$. Then, using \eqref{baza}, we write for more convenience in computations $X=U_{\alpha}+V_{\beta}$, where $U_{\alpha}=\alpha_1\tilde E_1+\alpha_2\tilde E_2-\alpha_3\tilde E_3$ and $V_{\beta}=\beta_1\tilde F_1+\beta_2\tilde F_2-\beta_3\tilde F_3$. Similarly, $Y=U_{\gamma}+V_{\delta}$. We now use the relations in \eqref{nablaJ} and compute
$$
G(U_{\alpha},V_{\beta})=\frac{2}{3\sqrt{3}}(U_{\alpha\times\beta}-V_{\alpha\times\beta}),\quad G(U_{\alpha},U_{\beta})=\frac{2}{3\sqrt{3}}(U_{\alpha\times\beta}+2V_{\alpha\times\beta}).
$$As $PU_{\alpha}=V_{\alpha}$, we obtain that 
$$
G(V_{\alpha},V_{\beta})=-\frac{2}{3\sqrt{3}}(V_{\alpha\times\beta}+2U_{\alpha\times\beta}).
$$Finally, by linearity we get the relation in \eqref{formulaG}.
\end{proof}

From now on we will restrict ourselves to $3$-dimensional Lagrangian submanifolds $M$ of $\mathbb{S}^3\times\mathbb{S}^3$.
It is known from \cite{dioosvranckenwang} and \cite{diooshuetal} that, as the pull-back of $T(\mathbb{S}^3\times \mathbb{S}^3)$ to $M$ splits into $TM\oplus JTM$, there are two  endomorphisms $A,B:TM\to TM$ such that the restriction $P|_{TM}$ of $P$ to the submanifold equals $A+JB$, that is $PX=AX+JBX,$ for all $X\in TM$. Note that the previous formula, together with the fact that $P$ and $J$ anti-commute, also determines $P$ on the normal space by $PJX=-JPX=BX-JAX.$  In addition, from the properties of $J$ and $P$ it follows that $A$ and $B$ are symmetric operators which commute and satisfy moreover that $A^2+B^2=Id$  (see \cite{dioosvranckenwang}). Hence $A$ and $B$ can be diagonalised simultaneously at a point $p$ in $M$ and there is an orthonormal basis $e_1, e_2,e_3\in T_pM$ such that 
\begin{equation}\label{Pei}
Pe_i=\cos (2\theta_i) e_i+\sin (2\theta_i) Je_i.
\end{equation}
The functions $\theta_i$ are called the angle functions of the immersion. 
Next, for a point $p$ belonging to an open dense subset of $M$ on which the multiplicities of the eigenvalues of $A$ and $B$ are constant (see \cite{29}),  we may extend the orthonormal basis $e_1,e_2,e_3$ to a frame on a neighborhood  in the Lagrangian submanifold. Finally, taking into account the properties of $G$ we know that there exists a local orthonormal frame $\{E_1,E_2,E_3\}$ on an open subset of $M$ such that  
\begin{equation}
AE_i=\cos(2\theta_i)E_i, \quad BE_i=\sin(2\theta_i)E_i
\end{equation}
and 
\begin{equation}\label{G}
\quad JG(E_i,E_j)=\frac{1}{\sqrt{3}}\varepsilon_{ijk}E_k.
\end{equation}
The following result is known (\cite{dioosvranckenwang}):
\begin{proposition}
The sum of the angles $\theta_1+\theta_2+\theta_3$ is zero modulo $\pi$.
\end{proposition}

For the Levi-Civita connection $\nabla$ on M we introduce (see\cite{dioosvranckenwang}) the functions $\omega_{ij}^k$ satisfying 
$$\nabla_{E_i}E_j=\sum\limits_{k=1}^3\omega_{ij}^k E_k \quad \text{ and } \quad \omega_{ij}^k=-\omega_{ik}^j.$$ 
Also, we denote by $h_{ij}^k$ the components of the cubic form on $M$:
\begin{equation}\label{totallysym}
h_{ij}^k= g( h(E_i,E_j),JE_k).
\end{equation}

\section{Results}\label{Results}

\subsection {Elementary properties of orientable minimal surfaces in  $\mathbb{S}^3$.\\}\label{minimalsurface}

We recall some elementary properties of minimal surfaces. Let $p: S \to  \mathbb{S}^3\subset \mathbb{R}^4$ be an oriented minimal surface. We are going to check that the immersion either admits local isothermal coordinates for which the conformal factor  satisfies the Sinh-Gordon equation or is totally geodesic. 
First, we take isothermal coordinates $ u, v$ such that $\partial u, \partial v$ is positively oriented, $\langle \partial u, \partial u \rangle =\langle \partial v, \partial v \rangle= 2 e^\omega$  and $\langle \partial u, \partial v \rangle=0$ in  a neighborhood of a point of $S$. As  it is often more useful to use complex notation we write $z=u+Iv$ and consider $\partial z=\frac{1}{2}(\partial u-I \partial v)$ and $\partial \bar{z}=\frac{1}{2}(\partial u+ I \partial v)$. Note that we use $I$ here in order to distinguish between the $i,j,k$ introduced in the quaternions. We also extend everything in a linear way in $I$. This means that $\langle \partial z, \partial z \rangle=\langle \partial \bar z, \partial \bar z \rangle=0$ and  $\langle \partial z, \partial \bar z \rangle= e^\omega$. If we write $\partial u= p \alpha$ and $\partial v= p \beta$, the unit normal is given by $N = p \tfrac{\alpha \times \beta}{2 e^\omega}$. It is elementary to check that this is independent of the choice of complex coordinate and that the matrix $\Big( p\  \tfrac{\partial u}{\vert \partial u \vert} \  \tfrac{\partial v}{\vert \partial v \vert} \  N  \Big)$ belongs to $SO(4)$. We denote by $\sigma$ the component of the second fundamental form  in the direction of $N$.  
Remark that with this choice, the minimality of the surface implies $\sigma(\partial z,\partial \bar z)=0$ and we may determine the components of the connection $\nabla$ on the surface:
\begin{equation}
\nabla_{\partial z} {\partial z}=\omega_z \partial z, \quad \nabla_{\partial z} {\partial \bar{z}}=\nabla_{\partial {\bar z}}{\partial z}=0 \quad \text{and} \quad  \nabla_{\partial\bar{z}}=\omega_{\bar{z}}\partial{\bar{z}}. 
\end{equation}
The Codazzi equation of a surface in  $\mathbb{S}^3$ states that
$$
\nabla \sigma(\partial z,\partial{\bar{ z}},\partial z)=\nabla \sigma(\partial{\bar{ z}},\partial z,\partial z).
$$ So it follows that $\partial\bar{z}(\sigma(\partial z,\partial z))=0.$
Hence $\sigma(\partial z,\partial z)$ is a holomorphic function.  Then we have two cases:\\
Case 1.
If $\sigma(\partial z, \partial z)=0$ on an open set, then by conjugation $\sigma(\partial {\bar{z}}, \partial \bar{ z})=0$ and therefore, using the analyticity of a minimal surface, $ \sigma=0$ everywhere.\\
Case 2.
If $\sigma(\partial z,\partial z)\neq 0$, then there exists  a function $g(z)$ such that $\sigma(\partial z,\partial z)=g(z)$. Away from isolated points we can always make a change of coordinates if necessary such that $\sigma(\partial z,\partial z)=-1$. Notice that by conjugation we get also $\sigma(\partial \bar z,\partial\bar z)=-1$.
Such a change of coordinates is unique up to translations and replacing $z$ by $-z$. \\
Next, given the immersions $p: S \rightarrow  \mathbb{S}^3(1)  \overset{i}{\hookrightarrow} \mathbb{R}^4$, from the Gauss formula we obtain:
\begin{align}\label{secondder}
&p_{zz}=\omega_z p_z-N,\nonumber\\
&p_{z\bar z}=-e^{\omega}p,\\
&p_{\bar z \bar z}=\omega_{\bar z} p_{\bar z}-N,\nonumber
\end{align}
where $N$ is the normal on  $\mathbb{S}^3$ and $N_z=e^{-\omega}p_{\bar z}$, $N_{\bar z}=e^{-\omega}p_z.$ Therefore
$$
p_{zz\bar z}=(\omega_{z\bar z}-e^{-\omega})p_z-\omega_z e^\omega p,\quad p_{z\bar z z}=-e^\omega \omega_z p-e^\omega p_z,
$$
which shows that $\omega$ satisfies 
\begin{align}\label{lap}
&\omega_{z\bar z}=-2 \sinh \omega\quad \Longleftrightarrow\nonumber\\
&\Delta \omega=-8\sinh \omega \quad \text{(Sinh-Gordon equation)}.
\end{align}
Notice that by $\Delta \omega$ we denote the Euclidean Laplacian of $\omega$ in $\R^2=\C$.\\
Let $\mathcal P$ be the lift of the minimal immersion to the immersion of the frame bundle in $SO(4)$, i.e.
$$\mathcal P : US\rightarrow SO(4): w \mapsto ( p \ w\ \tilde{J}w \  N ),$$
where $US$ denotes the unit tangent bundle of $S$ and $\tilde J$ denotes the natural complex structure on an orientable surface. In terms of our chosen isothermal coordinate this map can be parametrised by
$$
\mathcal P(u,v,t)=\left( p(u,v), \cos t\frac{p_u}{\mid p_u\mid}+\sin t\frac{p_v}{\mid p_v\mid}, -\sin t\frac{p_u}{\mid p_u\mid}+\cos t\frac{p_v}{\mid p_v\mid} , N (u,v)\right),
$$ for some real parameter $t$.
Note that we have the frame equations which state that
$$d \mathcal P =\mathcal P \Omega^t=-\mathcal  P \Omega, $$
where in terms of the coordinates $u$, $v$ and $t$ the matrix $\Omega$ is given 
by
\[
\begin{pmatrix}
0 &\sqrt{2}e^{\frac{\omega}{2}}(\cos (t) du+\sin (t) dv)&  \sqrt{2}e^{\frac{\omega}{2}}(\cos (t) dv-\sin (t) du)&0 \\
 \begin{tabular}{@{}c@{}}$- \sqrt{2}e^{\frac{\omega}{2}}(\cos (t) du+$\\$+\sin (t) dv)$\end{tabular}  & 0& \frac{1}{2}(\omega_u dv-\omega_v du)+dt &\begin{tabular}{@{}c@{}}$ -\sqrt{2}e^{-\frac{\omega}{2}}(\cos (t) du-$\\$-\sin (t) dv) $\end{tabular} \\
 \begin{tabular}{@{}c@{}}$-\sqrt{2}e^{\frac{\omega}{2}}(\cos (t) dv-$\\$-\sin (t) du)$\end{tabular}&   -\frac{1}{2}(\omega_u dv-\omega_v du)-dt& 0&   \begin{tabular}{@{}c@{}}$ \sqrt{2}e^{-\frac{\omega}{2}}(\sin (t) du+$\\$+\cos (t) dv)$\end{tabular}\\
0&\sqrt{2}e^{-\frac{\omega}{2}}(\cos (t) du-\sin (t) dv) &-\sqrt{2}e^{-\frac{\omega}{2}}(\sin (t) du+\cos (t) dv)&0\\
\end{pmatrix}. \]

\subsection{From the Lagrangian immersion to the minimal surface}\label{subs}\hfill \break

Now we will consider Lagrangian submanifolds in the nearly K\"ahler  $\mathbb{S}^3 \times \mathbb{S}^3$. We write 
the Lagrangian submanifold $M$  as 
\begin{align*}
&f:M\to \mathbb{S}^3\times\mathbb{S}^3\\
&x\mapsto f(x)=(p(x),q(x)),
\end{align*}
and we assume that the first component has nowhere maximal rank.
We have the following:
\begin{theorem}\label{th1}
 Let \begin{align*}
&f:M\to \mathbb{S}^3\times\mathbb{S}^3\\
&x\mapsto f(x)=(p(x),q(x)),
\end{align*}
be a Lagrangian immersion such that $p:M \rightarrow  \mathbb{S}^3$ has nowhere  maximal rank.
Then $\tfrac \pi 3$ is an angle function up to a multiple of $\pi$. 
The converse is also true.
\end{theorem}
\begin{proof} It is clear that $p$  has nowhere maximal rank if and only if there exists a non zero vector field $X$ such that $dp(X)=0$.
As usual we identify $df(X)$ with $X$, so we have that $X=df(X)= (dp(X),dq(X))$
and $QX= (-dp(X),dq(X))$. Therefore $p$ has nowhere maximal rank if and only if
\begin{align*}
X&=QX\\
&=\tfrac 1 {\sqrt{3}} (2 PJ X-JX)\\
&=\tfrac 1 {\sqrt{3}}(2 BX -2 JAX-JX).
\end{align*}
Comparing tangent and normal components we see that this is the case if and
only if
\begin{align*}
&AX= -\tfrac 12 X
&BX=\frac{\sqrt{3}} 2 X.
\end{align*}
So we see that $X$ is an eigenvector of both $A$ and $B$ and that the corresponding angle function is $\tfrac \pi 3$ (up to a multiple of $\pi$). 
\end{proof}
For the remainder of the paper we will consider Lagrangian immersions for which 
the map $p$ has nowhere maximal rank. In view of the previous lemma this means 
that one of the angle functions is constant, namely $\theta_1=\frac{\pi}{3}$. Then  using that the angles are only determined up to a multiple of $\pi$ and given that $2\theta_1+2\theta_2+2\theta_3$ is a multiple of $2\pi $, we may write
\begin{align}\label{unghi}
\begin{array}{l}
2\theta_1=\frac{2\pi}{3},\\
2\theta_2=2\Lambda+\frac{2\pi}{3},\\
2\theta_3=-2\Lambda+\frac{2\pi}{3},
\end{array}
\end{align}
for $\Lambda$ an arbitrary function. 
If necessary by interchanging $E_2,E_3$ with $-E_3,E_2$ we may assume that $\Lambda \ge 0$. Similarly if necessary interchanging $E_1,E_3$ by $-E_1,-E_3$ we may also assume that $h_{13}^3 \le 0$ (see equation \eqref{totallysym}). \\
In case that $\Lambda$ is $0$ or $\frac{\pi}{2}$  modulo $\pi$, we have that at least two of the angle functions coincide. It then follows from \cite{diooshuetal} that the Lagrangian submanifold is totally geodesic with angle function $\tfrac{\pi}{3}$ and that the immersion is congruent with  $f\colon \mathbb{S}^3\to \nks: u\mapsto (1,u)$.  Therefore in the remainder of this paper we will always assume that $ \Lambda \in(0,\frac{\pi}{2})$. \\
Notice that the case when $\Lambda$ is constant is treated in  \cite{alphaconstan-paperBurcu}, where it is shown that in that case $M$ is congruent with 
one of the following two immersions:
\begin{theorem}
Let~$M$ be a Lagrangian submanifold  of constant  sectional curvature in the nearly K\"ahler~$\nks$. If $M$ is not totally geodesic, 
then up to an isometry of
the nearly K\"ahler~$\nks$, ~$M$ is locally congruent with  one of the following
immersions:
\begin{enumerate}
\item  $f\colon \mathbb{S}^3\to \nks: u\mapsto (uiu^{-1},uju^{-1})$, 
\item $f: \mathbb R^3\to \nks: (u,v,w)\mapsto (p(u,w),q(u,v))$, where $p$ and $q$ are constant mean curvature tori in $\mathbb{S}^3$ given by
\begin{align*}
p(u,w)&=\left (\cos u \cos w,\cos u \sin w,\sin u \cos w,\sin u \sin w\right),\\
  q(u,v)&=\frac{1}{\sqrt{2}}\left (\cos v \left(\sin u+\cos u\right),\sin
   v \left(\sin u+\cos u\right)\right . ,\\
   &\qquad \left .  \cos v \left(\sin u-\cos u\right),\sin v \left(\sin
   u-\cos u\right) \right).
\end{align*} 
 \end{enumerate}
\end{theorem}
Note that these are precisely the two Lagrangian immersions with constant sectional curvature obtained is \cite{dioosvranckenwang}. These two examples will appear as special solutions in respectively Case 2 and Case 3. However we will mainly focus on the case that $\Lambda$ is not constant. \\
In the following, we will  identify a tangent vector $X$ in $T_xM$ with its image through $df$ in $T_{(p,q)}\mathbb{S}^3\times \mathbb{S}^3$, that is $X\equiv df(X)=(dp(X),dq(X))$, and we can write $QX\equiv Q(df(X))=(-dp(X),dq(X))$. Therefore, if we see $dp(X)$  projected on the first factor of $ \mathbb{S}^3\times  \mathbb{S}^3$ , that is $dp(X)\equiv (dp(X),0)$, we can write 
\begin{equation}
\label{dp}
dp(X)=\frac{1}{2}(X-QX).
\end{equation}
We use relations \eqref{Pei} and \eqref{unghi} to compute $PE_1=-\frac{1}{2}E_1+\frac{\sqrt{3}}{2}JE_1$. As mentioned before this is equivalent with stating that
$dp(E_1)= 0$ and that $p$ has nowhere maximal rank.
By straightforward computations we obtain 
\begin{equation} 
\label{15bis}
\begin{array}{l}
(dp(E_2),0)=\left(\frac{1}{2}-\frac{1}{\sqrt{3}}\sin(2\Lambda+\frac{2\pi}{3})\right )E_2+\frac{1}{\sqrt{3}}\left (\frac{1}{2}+\cos(2\Lambda+\frac{2\pi}{3})\right )JE_2,\\
(dp(E_3),0)=\left(\frac{1}{2}-\frac{1}{\sqrt{3}}\sin(-2\Lambda+\frac{2\pi}{3})\right )E_3+\frac{1}{\sqrt{3}}\left (\frac{1}{2}+\cos(-2\Lambda+\frac{2\pi}{3})\right )JE_3
 \end{array} \end{equation}
and 
\begin{align}
\label{length}
&\langle dp(E_2),dp(E_2)\rangle=\sin^2\Lambda,\nonumber\\
&\langle dp(E_3),dp(E_3)\rangle=\sin^2\Lambda,\\
&\langle dp(E_2),dp(E_3)\rangle=0.\nonumber
\end{align}
We denote
\begin{align}\label{vxi}
&v_2:=dp(E_2)\equiv(dp(E_2),0),\nonumber\\
&v_3:=dp(E_3)\equiv(dp(E_3),0),\\
&\xi=\frac{1}{\sqrt{3}}E_1-JE_1\nonumber
\end{align}
and we may easily see that $Q\xi=-\xi$, i.e. $\xi$ lies entirely on the first factor of $\mathbb{S}^3\times \mathbb{S}^3$. Moreover, 
 $\langle v_i,v_j\rangle=\delta_{ij}\sin\Lambda$, $\langle \xi,v_2 \rangle=\langle \xi,v_3 \rangle=0 $ and  $\langle \xi,\xi \rangle=1$. Therefore, $p(M)$ is a surface in $\mathbb{S}^3$ and $\xi$ can be seen as a unit normal to the surface. \\
As far as the Lagrangian immersion itself is concerned we also have due to the minimality that
\begin{align}
\label{minimality}
\begin{array}{l}
h_{11}^1+h_{12}^2+h_{13}^3=0,\\
h_{11}^2+h_{22}^2+h_{23}^3=0,\\
h_{11}^3+h_{22}^3+h_{33}^3=0.
\end{array}
\end{align}
From \cite{dioosvranckenwang} we know that the covariant derivatives of the endomorphisms $A$ and $B$ are
\begin{align}
(\nabla_XA)Y&=BS_{JX}Y-Jh(X,BY)+\frac{1}{2}(JG(X,AY)-AJG(X,Y)),\label{opA}\\
(\nabla_XB)Y&=-AS_{JX}Y+Jh(X,AY)+\frac{1}{2}(JG(X,AY)-AJG(X,Y)).\label{opB}
\end{align}
We are going to use the  definition of $\nabla A$ and $\nabla B$  in the previous expressions and then evaluate them for different vectors in the basis  in order to get information about the functions $\omega_{ij}^k$ and $h_{ij}^k$. For $X=Y=E_1$ in \eqref{opA} we obtain that 
\begin{align}\label{x25}
\begin{array}{l}
h_{12}^2=-h_{13}^3,\\
\omega_{11}^2=h_{11}^2 \cot\Lambda,\\
\omega_{11}^3=-h_{11}^3 \cot\Lambda.
\end{array}
\end{align}
If we take $X=E_1$ and $Y=E_2$ in \eqref{opA} and \eqref{opB}, we see that 
\begin{align}
&E_1(\Lambda)=h_{13}^3,\\
&\omega_{12}^3=\frac{\sqrt{3}}{6}-h_{12}^3\cot{2\Lambda}
\end{align}
and, for $X=E_2$ and $Y=E_1$ in \eqref{opA}, we obtain
\begin{align}
h_{11}^2&=0,\label{x28}\\
\omega_{21}^2&=-\cot\Lambda h_{13}^3,\\
\omega_{21}^3&=-\frac{\sqrt{3}}{6}-h_{12}^3 \cot\Lambda.
\end{align}
Then we choose successively $X=E_3,Y=E_1$, $X=E_2,Y=E_3$ and $X=E_3,Y=E_2$    in relations \eqref{opA} and \eqref{opB}
and obtain
\begin{align}
h_{11}^3&=0,\label{x31}\\
\omega_{31}^2&=\frac{\sqrt{3}}{6}+\cot\Lambda h_{12}^3,\\
\omega_{31}^3&=-\cot\Lambda h_{13}^3,\\
\omega_{22}^3&=-\cot 2\Lambda h_{22}^3,\\
\omega_{32}^3&=-\cot 2\Lambda h_{23}^3\\
E_2(\Lambda)&=h_{23}^3,\label{dere2}\\
E_3(\Lambda)&=-h_{22}^3\label{dere3}.
\end{align}
We can easily see from \eqref{x25}, \eqref{x28} and \eqref{x31} that 
$$
\omega_{11}^2=0\quad \text{and} \quad \omega_{11}^3=0 $$
and, if we consider as well the relations in \eqref{minimality}, we have that
$$
h_{33}^3=-h_{22}^3,\quad h_{11}^1=0\quad  \text{and} \quad h_{22}^2=-h_{23 }^3.
$$
Later on we will also need to study the Codazzi equations  for $M$. From \cite{dioosvranckenwang} we know their general form:
\begin{align}
\nabla h(X,Y,Z)-\nabla h(Y,X,Z)=&\frac{1}{3}(g(AY,Z)JBX-g(AX,Z)JBY\nonumber\\
	&-g(BY,Z)JAX+g(BX,Z)JAY ).
\end{align}
We are going to use the definition for $\nabla h$ in the previous relation and take different values for the vectors $X,Y$ and $Z$. Thus, we evaluate it successively for $E_1,E_2,E_1$; $E_1,E_2,E_2$; $E_1,E_3,E_3$; $E_1,E_3,E_2$ and $E_2,E_3,E_3$ and we obtain the following relations, respectively:
\begin{align}
\label{cdz0}
\begin{array}{l}
E_1(h_{13}^3)=\frac{1}{3}(-\sqrt{3}h_{12}^3+6(h_{13}^3)^2\cot\Lambda-6(h_{12}^3)^2\csc(2\Lambda)+\sin(2\Lambda)),\\
E_1(h_{12}^3)=\frac{1}{3}h_{13}^3(\sqrt{3}+9h_{12}^3\cot\Lambda+3h_{12}^3\tan\Lambda), 
\end{array}
\end{align}

\begin{align}
\label{cdz}
E_2(h_{13}^3)-E_1(h_{23}^3)=&\frac{1}{\sqrt{3}}h_{22}^3+h_{12}^3h_{22}^3\cot\Lambda-
h_{13}^3 h_{23}^3\cot\Lambda-h_{12}^3 h_{22}^3\cot(2\Lambda),\nonumber\\
E_1(h_{22}^3)-E_2(h_{12}^3)=&h_{13}^3h_{22}^3(2\cot\Lambda-\tan\Lambda)
+\frac{1}{6}h_{23}^3(2\sqrt{3}-3h_{12}^3\cot\Lambda+9h_{12}^3\tan\Lambda),\nonumber\\
E_3(h_{12}^3)-E_1(h_{23}^3)=&\frac{1}{\sqrt{3}}h_{22}^3+(h_{12}^3h_{22}^3-h_{13}^3h_{23}^3)\cot\Lambda
-(3h_{12}^3h_{22}^3+2 h_{13}^3h_{23}^3)\cot(2\Lambda),\nonumber\\
E_3(h_{13}^3)+E_1(h_{22}^3)=&\frac{1}{\sqrt{3}}h_{23}^3+h_{13}^3h_{22}^3\cot\Lambda+h_{12}^3h_{23}^3\cot\Lambda
-h_{12}^3h_{23}^3\cot(2\Lambda),\\
E_2(h_{13}^3)-E_3(h_{12}^3)=&2(h_{12}^3h_{22}^3+h_{13}^3h_{23}^3)\cot(2\Lambda),\nonumber\\
E_3(h_{22}^3)-E_2(h_{23}^3)=&-\frac{1}{2}(8(h_{12}^3)^2+4(h_{13}^3)^2+3((h_{22}^3)^2+(h_{23}^3)^2))\cot\Lambda- \nonumber\\
&-\frac{1}{3}(\sqrt{3}h_{12}^3+\sin{4\Lambda})+\frac{3}{2}((h_{22}^3)^2+(h_{23}^3)^2)\tan\Lambda,\notag \nonumber\\
E_2(h_{22}^3)+E_3(h_{23}^3)=&-\frac{1}{3}h_{13}^3(\sqrt{3}+6h_{12}^3\cot\Lambda). \nonumber
\end{align}
\begin{theorem} \label{th3}
 Let \begin{align*}
&f:M\to \mathbb{S}^3\times\mathbb{S}^3\\
&x\mapsto f(x)=(p(x),q(x)),
\end{align*}
be a Lagrangian immersion such that $p:M \rightarrow  \mathbb{S}^3$ has nowhere maximal rank.  
Assume that $M$ is not totally geodesic. Then $p(M)$ is a (branched) minimal surface in  $\mathbb{S}^3$.
Moreover
$$
\tilde{P}:M\rightarrow SO(4):\ x\mapsto\left( p(x)\quad\frac{v_2}{\sin\Lambda} \quad \frac{v_3}{\sin\Lambda} \quad \xi \right),
$$ where $v_2,v_3$ and $\xi$ are defined by \eqref{vxi}, is a map which is contained into the frame bundle over the minimal surface $p$.
\end{theorem}
\begin{proof}
As mentioned before, by restricting to an open dense subset of $M$, we may assume that $\Lambda\in (0, \frac{\pi}{2})$. Recall that $dp(E_1)=0,$ hence $p(M)$ is a surface. Denoting the second fundamental form of the surface in the direction of $\xi$ by $\sigma$, a straightforward computation yields that
\begin{equation} 
\begin{array}{l}
\sigma(E_2,E_2)=h_{13}^3,\\
\sigma(E_2,E_3)=\sigma(E_3,E_2)=\frac{1}{\sqrt{3}}\cos\Lambda \sin \Lambda-h_{12}^3,\\
\sigma(E_3,E_3)=-h_{13}^3.
\end{array}
\end{equation}
As $dp(E_2)$ and $dp(E_3)$ are orthogonal and have the same length, the above formulas indeed imply that the surface is minimal.\\ 
Moreover we also see that the surface is totally geodesic if and only if $h_{13}^3=0$ and $h_{12}^3=\frac{1}{\sqrt{3}}\cos\Lambda \sin \Lambda$. 
Note also that if we write $(dp(E_2),0)=(p \alpha,0)$ and $(dp(E_3),0)=(p \gamma,0)$, we have that
\begin{align*}
G((dp(E_2),0),(dp(E_3),0))&=G((p \alpha,0),(p \gamma,0))\\
&=\tfrac 2 {3\sqrt{3}}(p(\alpha \times \gamma), 2 q (\alpha \times \gamma)).
\end{align*}
Therefore,
$$(p(\alpha \times \gamma),0)=\tfrac{3 \sqrt{3}}{4}(G((dp(E_2),0),(dp(E_3),0)- Q(G((dp(E_2),0),(dp(E_3),0))).$$
A straightforward computation, using \eqref{15bis} and \eqref{G}, shows that
this gives
$$(p(\alpha \times \gamma),0)=(\sin \Lambda)^2 \xi.$$
Therefore $\xi$ corresponds with the normal $N$ on the surface. 
\end{proof}

\subsection{The reverse construction} \label{section3.3}\hfill \break 
In the following, we will separate the study of the submanifold into three cases, according to whether the surface is totally geodesic or not and whether the map to the frame bundle is an immersion or not. 

\subsubsection{\textbf{Case 1. $p(M)$ is not a totally geodesic surface and the map $\tilde{\mathcal  P}$ is an immersion}}
In that case we can identify $M$ with the frame bundle on the minimal surface induced earlier. Recall that 
$$
\tilde{ \mathcal P}: x\in M \mapsto \left(p \ \frac{v_2}{\sin\Lambda} \ \frac{v_3}{\sin\Lambda} \ \xi \right ).
$$\\
Writing again $d \tilde{ \mathcal  P} =-\tilde{ \mathcal P} \tilde\Omega, $ we can express the matrix $\tilde\Omega$ in terms of $\{ E_1,E_2,E_3\}$ by
\[
\begin{pmatrix}
0& \sin(\Lambda)\omega_2  & \sin(\Lambda)\omega_3 &0 \\%
-\sin(\Lambda)\omega_2 & 0& \begin{tabular}{@{}c@{}}$\left(\frac{1}{\sqrt{3}}+h_{12}^3 \csc(2\Lambda)\right)\omega_1+$\\$+h_{22}^3\csc(2\Lambda)\omega_2+$\\$+h_{23}^3 \csc(2\Lambda)\omega_3$\end{tabular} &\begin{tabular}{@{}c@{}}$ h_{13}^3 \csc(\Lambda)\omega_2 +$\\$+\left(\frac{\cos\Lambda}{\sqrt{3}}-h_{12}^3 \csc\Lambda\right)\omega_3 $\end{tabular}\\%
-\sin(\Lambda)\omega_3&  \begin{tabular}{@{}c@{}}$ - (\frac{1}{\sqrt{3}}+h_{12}^3 \csc(2\Lambda))\omega_1-$\\$-h_{22}^3 \csc(2\Lambda)\omega_2-h_{23}^3 \csc(2\Lambda)\omega_3$\end{tabular} & 0&  \begin{tabular}{@{}c@{}}$\left(\frac{\cos\Lambda}{\sqrt{3}}-h_{12}^3 \csc(\Lambda)\right)\omega_2-$\\$-h_{13}^3 \csc(\Lambda)\omega_3$\end{tabular}\\%
0& \begin{tabular}{@{}c@{}}$\left(-\frac{\cos\Lambda}{\sqrt{3}}+h_{12}^3 \csc\Lambda\right)\omega_3-$\\$- h_{13}^3 \csc(\Lambda) \omega_2$\end{tabular}  &  \begin{tabular}{@{}c@{}}$\left(-\frac{\cos\Lambda}{\sqrt{3}}-h_{12}^3 \csc(\Lambda)\right)\omega_2-$\\$-h_{13}^3 \csc(\Lambda)\omega_3$\end{tabular}&0
\end{pmatrix}, \]
where $\omega_i(E_i)=\delta_{ij}.$
The above matrix implies that the map $\tilde{\mathcal P}$ into $SO(4) \subset \mathbb R^{16}$ is an immersion if and only if 
$$\tfrac{1}{\sqrt{3}} +h_{12}^3 \csc(2 \Lambda) \ne 0
.$$ 
As it is an immersion, in view of the dimensions, its image is an open part of the frame bundle and we can identify $M$ with an open part of the frame bundle on the minimal surface. Moreover we can write
$$\frac{v_2}{\sin\Lambda}= \cos (t+\gamma(t,u,v))\frac{p_u}{\mid p_u\mid}+\sin (t+\gamma(t,u,v))\frac{p_v}{\mid p_v\mid},$$
where $\gamma$ is some function. As $\tilde{\mathcal P}$ is an immersion, we have that $t +\gamma(t,u,v)$ depends on $t$ and can be taken as the new variable $t$ on the frame bundle. Doing so, we have that $\tilde{\mathcal P} = \mathcal P$ and $\tilde \Omega = \Omega$ (for $\mathcal P, \Omega$ as in subsection \ref{minimalsurface}). Comparing both expressions for the matrix $\Omega$ we deduce
\begin{align*} 
\omega_1=&\frac{1}{\frac{1}{\sqrt{3}}+h_{12}^3\csc 2\Lambda}\left(-\left( \sqrt{2}\frac{\csc 2\Lambda}{\sin \Lambda} e^{\omega/2}(h_{22}^3\cos t- h_{23}^3\sin t)+\frac{1}{2}\omega_v  \right)du-\right.\\
	&\left. -\left( \sqrt{2}\frac{\csc 2\Lambda}{\sin \Lambda} e^{\omega/2}(h_{22}^3\sin t+h_{23}^3\cos t )-\frac{1}{2}\omega_u  \right)dv+dt\right),\\
\omega_2=&\frac{1}{\sin\Lambda}\sqrt{2} e^{\omega/2}(\cos( t)du+\sin(t)dv),\\
\omega_3=&\frac{1}{\sin\Lambda}\sqrt{2} e^{\omega/2}(\cos( t)dv-\sin(t)du),
\end{align*}
as well as 
\begin{align}
\left\{
\begin{array}{ll}
&e^{-\omega}\cos(2t)+h_{13}^3\frac{1}{\sin^2\Lambda}=0,\\
&e^{-\omega}\sin(2t)+\left(h_{12}^3\csc\Lambda-\frac{\cos\Lambda}{\sqrt{3}}\right)\frac{1}{\sin\Lambda}=0,
\end{array}
\right.
\end{align}
which implies that 
\begin{align}
\label{h}
\left\{
\begin{array}{ll}
h_{13}^3&=-e^{-\omega}\cos(2t)\sin^2\Lambda,\\
h_{12}^3&=\left(-e^{-\omega}\sin(2t)\sin\Lambda  +\frac{\cos\Lambda}{\sqrt{3}}  \right)\sin\Lambda.
\end{array}
\right.
\end{align}
We may express $E_1,E_2,E_3$  with respect to the basis $\{\partial t,\partial u,\partial v\}$ as follows. For $E_i=a_i \partial t+b_i \partial u+c_i \partial v$, we use the previously obtained expressions of $\omega_j$ in $\omega_j(E_i)=\delta_{ij}$ and by straightforward computations we get
\begin{align}
E_1=&\left(\frac{1}{\sqrt{3}}+h_{12}^3\csc(2\Lambda)\right)\partial t,\nonumber\\
E_2=&(\csc(2\Lambda)h_{22}^3+\frac{1}{2\sqrt{2}}\sin\Lambda e^{-\omega/2}(\cos( t)\omega_v-\sin(t)\omega_u)  )\partial t+\nonumber\\
& +\frac{e^{-\omega/2}\cos t \sin\Lambda}{\sqrt{2}} \partial u+ \frac{e^{-\omega/2}\sin t \sin\Lambda}{\sqrt{2}}\partial v,
\label{basisE}\\
E_3=&\left(\csc(2\Lambda)h_{23}^3-\frac{1}{2\sqrt{2}}\sin\Lambda \ e^{-\omega/2}(\cos( t)\omega_u+\sin(t)\omega_v)\right  )\partial t-\nonumber\\
&-\frac{e^{-\omega/2}\sin t \sin\Lambda}{\sqrt{2}}\partial u +\frac{e^{-\omega/2}\cos t \sin\Lambda}{\sqrt{2}}    \partial v.\nonumber
\end{align}
In order to be able to proceed with the reverse construction, i.e. in order to be able to construct a Lagrangian immersion starting from the minimal surface we need to express $\Lambda$, $h_{22}^3$ and $h_{23}^3$ in terms of the variables $t,u,v$. 
Remark that, as  $E_1(\Lambda)=h_{13}^3$, we may use \eqref{h} and the expression of $E_1$ in \eqref{basisE} to  determine how  $\Lambda$ depends on the variable $t$. We get
\begin{equation}\label{Lambdat}
\Lambda_t=-\frac{2\cos( 2t)\sin^2\Lambda}{\sqrt{3}e^\omega-2\cos t\sin t\tan\Lambda}.
\end{equation}
In order to solve the above differential equation, we use \eqref{Lambdat} to compute the derivative of the expression $ \frac{\sqrt{3}e^{\omega}}{\tan\Lambda}-\sin (2t)$: 
$$
\partial t\left(  \frac{\sqrt{3}e^{\omega}}{\tan\Lambda}-\sin (2t)\right)^2=2\sin (4t),
$$ which, by integration, gives
$
\left(  \frac{\sqrt{3}e^{\omega}}{\tan\Lambda}-\sin( 2t)\right)^2=-\frac{1}{2}\cos (4t)+\tfrac{c_1}{4},
$ where $c_1$ does not depend on $t$. 
 Notice that this implies
\begin{equation}
\label{tan}
\tan\Lambda=\frac{2\sqrt{3}e^\omega}{ \varepsilon_1\sqrt{ c_1-2\cos(4t)} +2\sin(2t)},
\end{equation}
where $\varepsilon_1 = \pm 1$ and, at the same time, the surface is defined on an open set where $c_1-2\cos(4t)\geq0$. Note that as the above expression contains a square root which would complicate simplifications later on, we will avoid its use as much as possible. 
For later use, remark that we can write
\begin{equation}\label{43}
\left( \frac{2\sqrt{3}e^{\omega}}{\tan\Lambda}-2\sin(2t)  \right)^2=c_1-2\cos (4t).
\end{equation}
This implies that $c_1 \ge -2$, where equality can hold only on a set of measure $0$ of the Lagrangian submanifold $M$. So on an open dense subset we can write
$$c_1 = e^{\omega+\mu}-2.$$
Combining this with the previous expression of $c_1$ and taking the derivative with respect to $u$ and $v$, we can compute 
\begin{align*}
&\Lambda_u=-\frac{\sin ^2\Lambda  \left(\mu_u+e^{\omega } \cot \Lambda 
   \left(3 e^{\omega } \cot \Lambda  (\mu_u-\omega_u)-2 \sqrt{3} \mu_u \sin (2 t)\right)+\omega_u\right)}{6 e^{2 \omega } \cot \Lambda -2 \sqrt{3} e^{\omega } \sin
   (2 t)}\\
&\Lambda_v=-\frac{\sin ^2\Lambda  \left(\mu_v+e^{\omega } \cot \Lambda 
   \left(3 e^{\omega } \cot \Lambda (\mu_v-\omega_v)-2 \sqrt{3} \mu_v \sin (2 t)\right)+\omega_v\right)}{6 e^{2 \omega } \cot \Lambda -2 \sqrt{3} e^{\omega } \sin
   (2 t)}.
\end{align*}

Using this, together with \eqref{basisE}, we can solve in \eqref{dere2} and \eqref{dere3}, for $h_{22}^3$ and $h_{23}^3$. This gives us
\begin{align*}
h_{22}^3=\frac{e^{-3 \omega /2} \sin ^2\Lambda }{6 \sqrt{2}}& \bigl(3 e^{\omega } \cos \Lambda 
   ((\omega_u-\mu_u) \sin t+(\mu_v-\omega_v) \cos t)-\bigr.\\
	&\bigl.-\sqrt{3} \sin \Lambda ((\mu_u+\omega_u) \cos (3 t)+(\mu_v+\omega_v)\sin (3 t))\bigr),\\
h_{23}^3=\frac{e^{-3 \omega /2} \sin ^2\alpha }{6 \sqrt{2}}& \bigl(\sqrt{3} \sin \Lambda 
   ((\mu_u+\omega_u) \sin (3 t)+(-\mu_v-\omega_v) \cos (3 t))-\bigr.\\
	&\left.-3 e^{\omega } \cos \Lambda 
   (\mu_u-\omega_u) \cos t-3 e^{\omega } \cos \Lambda
    (\mu_v-\omega_v) \sin t\right).
\end{align*}
In order to determine a differential equation for the function $\mu$ 
we now apply the previously obtained Codazzi equations  for $M$.  
 By \eqref{basisE}, it turns out that  \eqref{cdz0} and the first $5$ equations of \eqref{cdz} are trivially satisfied. 
 Recall from \eqref{lap} that $\Delta\omega=-8\sinh\omega$.  The seventh equation of \eqref{cdz} reduces to
 \begin{equation}
\label{diff2}
\Delta{\mu}=-4 e^{\omega } (\cos (2 \Lambda )+2) \csc ^2\Lambda +8 \sqrt{3} \cot
   \Lambda  \sin (2 t)+8 \sinh \omega .
\end{equation}
 A straightforward computation, using the definition of $\mu$  and \eqref{43},  shows that this reduces to
\begin{equation}
\label{diff}
\Delta{\mu}=-e^{{\mu}}.
\end{equation}
Further on, with these new notations,  we may see by straightforward computations that the sixth  equation  of \eqref{cdz} is now  trivially satisfied. 

\hfill\break
\textbf{ Reverse construction} \hfill\break
We denote by $p:S\rightarrow \mathbb{S}^3\subset\mathbb{R}^4$ a given minimal surface $S$ which is not totally geodesic, on which we take suitable isothermal coordinates as introduced before. Hence we have a solution $\omega$ of
$\Delta\omega=-8\sinh\omega$. Additionally, we take a solution of
\begin{equation}
\label{diff}
\Delta{\mu}=-e^{{\mu}}
\end{equation}
and we take the open part of the frame bundle such that
\begin{equation}
\left( \frac{2\sqrt{3}e^{\omega}}{\tan\Lambda}-2\sin(2t)  \right)^2=e^{\omega+\mu}-2-2\cos( 4t)
\end{equation}
has a solution for the function $\Lambda$ on an open domain. 
We define
\begin{align*}
h_{13}^3&=-e^{\omega}\cos(2t)\sin^2\Lambda,\\
h_{12}^3&=(-e^{-\omega}\sin(2t)\sin\Lambda+\frac{\cos\Lambda}{\sqrt{3}})\sin\Lambda,
\end{align*} 
\begin{align*}
h_{22}^3=\frac{e^{-3 \omega /2} \sin ^2\Lambda }{6 \sqrt{2}}& \bigl(3 e^{\omega } \cos \Lambda 
   ((\omega_u-\mu_u) \sin t+(\mu_v-\omega_v) \cos t)-\bigr.\\
	&\bigl.-\sqrt{3} \sin \Lambda ((\mu_u+\omega_u) \cos (3 t)+(\mu_v+\omega_v)\sin (3 t))\bigr),\\
h_{23}^3=\frac{e^{-3 \omega /2} \sin ^2\Lambda }{6 \sqrt{2}}& \bigl(\sqrt{3} \sin \Lambda 
   ((\mu_u+\omega_u) \sin (3 t)+(-\mu_v-\omega_v) \cos (3 t))-\bigr.\\
	&\left.-3 e^{\omega } \cos \Lambda
   (\mu_u-\omega_u) \cos t-3 e^{\omega } \cos \Lambda (\mu_v-\omega_v) \sin t\right)
\end{align*} 
and we define as well a metric on the open part of the frame bundle, by assuming that the 
vectors
\begin{align}
E_1=&\frac{1}{2} \left(\sqrt{3}-2 e^{-\omega } \tan \Lambda  \sin t \cos t\right)\partial t,\nonumber \\ 
E_2=&-\frac{e^{-3 \omega /2} \sin \Lambda  }{12 \sqrt{2}}  \Bigl( \sqrt{3} \tan \Lambda 
   ((\mu_u+\omega_u) \cos (3 t)+(\mu_v+\omega_v) \sin (3 t))+3 e^{\omega } ((\mu_u+\omega_u) \sin t+\Bigr. \nonumber\\
 &\Bigl. + (-\mu_v-\omega_v) \cos \ t)\Bigr) \partial t+\frac{e^{-\frac{\omega}{2}}\cos t \sin \Lambda}{\sqrt{2}} \partial u+  \frac{ e^{-\frac{\omega}{2}} \sin t \sin\Lambda }{\sqrt{2}} \partial v, \label{framerelation}      \\
E_3=&\frac{e^{-3 \omega /2} \sin \Lambda } {12 \sqrt{2}}\Bigl(\sqrt{3} \tan \Lambda
   (({\mu_u}+{\omega_u}) \sin (3 t)+(-{\mu_v}-{\omega_v}) \cos (3 t))-3 e^{\omega } (({\mu_u}+{\omega_u}) \cos t +\Bigr. \nonumber \\  &\Bigl.+({\mu_v}+{\omega_v}) \sin t)\Bigr)\partial t 
	-\frac{e^{-\frac{\omega}{2}}\sin t \sin \Lambda}{\sqrt{2}} \partial u+  \frac{ e^{-\frac{\omega}{2}} \cos t 			 \sin\Lambda }{\sqrt{2}} \partial v\nonumber
\end{align}
form an orthonormal basis. \\
We now want to determine the Lagrangian immersion 
\begin{align*}
&f:S\times I\rightarrow\mathbb{S}^3\times\mathbb{S}^3\\
&(u,v,t)\mapsto f(u,v,t)=(p(u,v,t),q(u,v,t)).
\end{align*}
We already know that the first component is the given minimal surface $p$. We write for both bases
\[ 
 \begin{array}{ccc}  \label{system}
 \frac{\partial}{\partial t}(q)=q\beta_1, &  \frac{\partial}{\partial t}(p)=p\alpha_1, \\
 \frac{\partial}{\partial u}(q)=q\beta_2, &  \frac{\partial}{\partial u}(p)=p\alpha_2, \\
 \frac{\partial}{\partial v}(q)=q\beta_3, &  \frac{\partial}{\partial v}(p)=p\alpha_3,
 \end{array}  
\text{and}
  \begin{array}{ccc}
E_1(q)=q\tilde{\beta}_1, & E_1(p)=p\tilde\alpha_1, \\
E_2(q)=q\tilde{\beta}_2, & E_2(p)=p\tilde\alpha_2, \\ 
E_3(q)=q \tilde{\beta}_3, & E_3(p)=p \tilde\alpha_3.
 \end{array} \]
Note that $\alpha_1=0$ and $\alpha_2$ and $\alpha_3$ are determined by the minimal surface. In particular $\alpha_2$ and $\alpha_3$ are mutually orthogonal imaginary quaternions with length squared $2 e^\omega$.
From \eqref{framerelation} we then get that
\begin{align*}
&\tilde \alpha_1=0,\\
&\tilde \alpha_2=\frac{e^{-\frac{\omega}{2}}\cos t \sin \Lambda}{\sqrt{2}} \alpha_2+  \frac{ e^{-\frac{\omega}{2}} \sin t \sin\Lambda }{\sqrt{2}} \alpha_3,\\
 &\tilde \alpha_3=-\frac{e^{-\frac{\omega}{2}}\sin t \sin \Lambda}{\sqrt{2}} \alpha_2+  \frac{ e^{-\frac{\omega}{2}} \cos t 			 \sin\Lambda }{\sqrt{2}} \alpha_3
\end{align*}
and from the properties of the minimal surface we obtain
\begin{align*}
&\tfrac{\partial \alpha_2}{\partial u}=-\tfrac{\partial \alpha_3}{\partial v}= \tfrac 12 \omega_u \alpha_2  -\tfrac 12 \omega_v \alpha_3 -e^\omega \alpha_2 \times \alpha_3,\\
&\tfrac{\partial \alpha_2}{\partial v}=\tfrac 12 \omega_v \alpha_2 +\tfrac 12 \omega_u \alpha_3 +\alpha_2 \times \alpha_3,\\
&\tfrac{\partial \alpha_3}{\partial u}=\tfrac 12 \omega_v \alpha_2 +\tfrac 12 \omega_u \alpha_3 -\alpha_2 \times \alpha_3.
\end{align*}
Using the properties of the vector cross product, this also implies that
\begin{align*}
&\tfrac{\partial \alpha_2\times \alpha_3}{\partial u}= 2  \alpha_2 +2 e^\omega \alpha_3+\omega_u \alpha_2 \times \alpha_3,\\
&\tfrac{\partial \alpha_2 \times \alpha_3}{\partial v}=-2 e^\omega \alpha_2 -2 \alpha_3+ \omega_v \alpha_2 \times \alpha_3.
 \end{align*}
Now, in order to find $\tilde \beta_i$, we remark that the vectors $E_1$, $E_2$ and $E_3$ need to correspond with eigenvectors of the operators $A$ and $B$ with suitable eigenfunctions. 
We have
\begin{align}
\label{expressionbasisE}
E_1&=(0,q\tilde\beta_1),\nonumber\\
E_2&=(p\tilde\alpha_2,q\tilde\beta_2),\\
E_3&=(p\tilde\alpha_3,q\tilde\beta_3).\nonumber
\end{align}
The angle functions are $\theta_1=\frac{2\pi}{3}, \theta_2=2\Lambda+\frac{2\pi}{3}, \theta_3=-2\Lambda+\frac{2\pi}{3}$ and 
\begin{align}
\label{PEE2}
PE_i=\cos(2\theta_i)E_i+\sin(\theta_i )JE_i,
\end{align}for $i=1,2,3$.
At the same time, by the definition of $P$ in \eqref{defP} and by \eqref{expressionbasisE} we have 
\begin{equation}\label{PP2}
PE_1=(p\tilde\beta_1,0), \quad PE_2=(p\tilde\beta_2,q\tilde\alpha_2), \quad PE_3=(p\tilde\beta_3,q\tilde\alpha_3).
\end{equation}
Now we use the definition of $J$ to write out $JE_i$:
\begin{align}\label{JEi2}
JE_1&=\frac{1}{\sqrt{3}}(2p\tilde\beta_1,q\tilde\beta_1),\nonumber\\
JE_2&=\frac{1}{\sqrt{3}}(p(2\tilde\beta_2-\tilde\alpha_2),q(-2\tilde\alpha_2+\tilde\beta_2)),\\
JE_3&=\frac{1}{\sqrt{3}}(p(2\tilde\beta_3-\tilde\alpha_3),q(-2\tilde\alpha_3+\tilde\beta_3)).\nonumber
\end{align}
Then, by  using \eqref{JEi2}, \eqref{expressionbasisE} and the values of $\theta_i$ in \eqref{unghi}, we rewrite equation \eqref{PEE2} and,  by comparing it to \eqref{PP2}, we obtain 
\begin{align*}
&\tilde\beta_2=\frac{\cos(2\Lambda+\frac{2\pi}{3})-\frac{1}{\sqrt{3}} \sin(2\Lambda+\frac{2\pi}{3})}{1-\frac{2}{\sqrt{3}}\sin(2\Lambda+\frac{2\pi}{3})}\tilde\alpha_2=\tfrac 12 (1-\sqrt{3} \cot \Lambda) \tilde \alpha_2, \\
&\tilde\beta_3=\frac{\cos(-2\Lambda+\frac{2\pi}{3})-\frac{1}{\sqrt{3}} \sin(-2\Lambda+\frac{2\pi}{3})}{1-\frac{2}{\sqrt{3}}\sin(-2\Lambda+\frac{2\pi}{3})}\tilde\alpha_3=\tfrac 12 (1+\sqrt{3} \cot \Lambda) \tilde \alpha_3.
\end{align*}
Next we continue the computations  in order to determine $\tilde\beta_1$. 
For this,  we compute $G(E_2,E_3)$ in two different ways, once using \eqref{G} and once using \eqref{formulaG}. We obtain, respectively
$$
G(E_2,E_3)=-\frac 1 {\sqrt{3}} JE_1=-\frac{1}{3}(p 2 \tilde \beta_1,q \tilde \beta_1),
$$
and 
\begin{align*}
G(E_2,E_3)&=G((p \tilde \alpha_2,q \tilde \beta_2),(p \tilde \alpha_3,q \tilde \beta_3))\\
&=\tfrac{2}{3 \sqrt{3}} (p(\tilde \beta_2  \times \alpha_3 +\tilde \alpha_2 \times \tilde \beta_3+ \tilde \alpha_2 \times \tilde \alpha_3 -2 \tilde \beta_2 \times \tilde \beta_3,\\
&\qquad \qquad q(-\tilde \beta_2  \times \alpha_3 -\tilde \alpha_2 \times \tilde \beta_3+ 2\tilde \alpha_2 \times \tilde \alpha_3 - \tilde \beta_2 \times \tilde \beta_3)\\
&=\tfrac{2}{3 \sqrt{3}}(p(2-\tfrac 12(1-3 \cot^2 \Lambda)) \tilde \alpha_2 \times \tilde \alpha_3,q(1-\tfrac 13(1-3 \cot^2 \Lambda)) \tilde \alpha_2 \times \tilde \alpha_3)\\
&=\tfrac{1}{2 \sqrt{3}}(1+\cot^2 \Lambda) (2 p  \tilde \alpha_2 \times \tilde \alpha_3,q  \tilde \alpha_2 \times \tilde \alpha_3).
\end{align*}
Hence, comparing both expressions we get that 
$$\tilde \beta_1= -\tfrac{\sqrt{3}}{2} \csc^2 \Lambda \ \tilde \alpha_2 \times \tilde \alpha_3=-\tfrac{\sqrt{3}}{4} e^{-\omega} \alpha_2 \times \alpha_3.$$
Moreover, we also obtain that
\begin{align*}
\tilde \beta_2&=\tfrac 1{2\sqrt{2}} (1-\sqrt{3} \cot \Lambda) e^{-\tfrac{\omega}{2}} \sin \Lambda  (\cos t \alpha_2 + \sin t \alpha_3),\\
\tilde \beta_3&=\tfrac 1{2\sqrt{2}} (1+\sqrt{3} \cot \Lambda) e^{-\tfrac{\omega}{2}} \sin \Lambda  (-\sin t \alpha_2 + \cos t \alpha_3).
\end{align*}
We then take the inverse of \eqref{framerelation} and  deduce that
\begin{align*}
\beta_1&= -\frac{\sqrt{3} \alpha_2 \times \alpha_3 }{2 \sqrt{3} e^{\omega }-2 \sin
   (2t) \tan (\Lambda )}, \\
\beta_2&= \frac{1}{8} \left(e^{-\omega } \left( {\mu_v}+ {\omega_v}-\frac{( {\mu_u}+ {\omega_u}) \cos
   (2 t) \tan (\Lambda )}{\sqrt{3} e^{\omega }
  -\sin (2 t) \tan (\Lambda )}\right)\alpha_2 \times \alpha_3  - 4 (\sqrt{3} \cot (\Lambda )\cos (2 t)+1) \alpha_2  -4 \sqrt{3}\sin (2 t) \cot\Lambda\alpha_3   \right), \\
 \beta_3&= \frac{1}{8} \left(-e^{-\omega } \left( {\mu_u}+ {\omega_u}+\frac{( {\mu_v}+ {\omega_v}) \cos
   (2 t) \tan (\Lambda )}{\sqrt{3} e^{\omega }
  -\sin (2 t) \tan (\Lambda )}\right)\alpha_2 \times \alpha_3  - 4 \sqrt{3} \cot (\Lambda )\sin (2 t) \alpha_2  +4(1+ \sqrt{3}\cos (2 t) \cot\Lambda)\alpha_3   \right).
\end{align*}
By straightforward computations, it now follows that
\begin{align*}
&\frac{\partial \beta_1}{\partial u}-\frac{\partial \beta_2}{\partial t}-2 \beta_1 \times \beta_2 =0,\\
&\frac{\partial \beta_1}{\partial v}-\frac{\partial \beta_3}{\partial t}-2 \beta_1 \times \beta_3 =0,\\
&\frac{\partial \beta_3}{\partial u}-\frac{\partial \beta_2}{\partial v}-2 \beta_3 \times \beta_2 =0,
\end{align*}
from which we deduce that the integrability conditions for the immersion $q$ are satsified.

\subsubsection{\textbf{Case 2. The minimal surface $p(M)$ is totally geodesic, i.e. $ \sigma=0$}}

As mentioned before this means that $h_{13}^3=0$, $ h_{12}^3=\frac{\cos\Lambda\sin\Lambda}{\sqrt{3}}$.
The equations following from  \eqref{opA} and \eqref{opB}, just like in the first case, give
\begin{equation}
\begin{array}{lll}
h_{12}^2=0,& \omega_{11}^2=0,&\omega_{21}^3=-\frac{2+\cos (2\Lambda)}{2\sqrt{3}},\\
h_{11}^2=0,&\omega_{11}^3=0,& \omega_{22}^3=-h_{22}^3\cot (2\Lambda), \\
h_{11}^3=0,&\omega_{12}^3=\frac{\sin^2\Lambda}{\sqrt{3}}, &\omega_{31}^2=\frac{2+\cos (2\Lambda)}{2\sqrt{3}},\\
\omega_{21}^2=0,&\omega_{31}^3=0, &\omega_{32}^3=-h_{23}^3\cot (2\Lambda) \\
\end{array}
\end{equation}
and 
\begin{equation}
\begin{array}{l}
E_1(\Lambda)=0,\\
E_2(\Lambda)=h_{23}^3,\\
E_3(\Lambda)=-h_{22}^3.
\end{array}
\end{equation}
In this case, the equations of Codazzi become
\begin{equation}\label{rel1}
E_1(h_{23}^3)=-\frac{\sqrt{3}}{2}h_{22}^3, \quad E_1(h_{22}^3)=\frac{\sqrt{3}}{2}h_{23}^3,\quad E_2(h_{22}^3)=-E_3(h_{23}^3)
\end{equation}and
\begin{equation}\label{rel1prim}
-1-(1+12 (h_{22}^3)^2+12 (h_{23}^3)^2)\cos (2\Lambda)+\cos(4\Lambda)+\cos(6\Lambda)+4(E_2(h_{23}^3)-E_3(h_{22}^3))\sin(2\Lambda)=0.
\end{equation}
In what follows we are going to introduce new vector fields on $M$ by:
\begin{align}
\label{newbasis}
&X_1=\frac{4}{\sqrt{3}} E_1, \nonumber\\
&X_2=-\frac{2 h_{22}^3 \csc ^2\Lambda \sec \Lambda }{\sqrt{3}}E_1+2
   \csc \Lambda \ E_2,\\
&X_3=-\frac{2 h_{23}^3 \csc ^2\Lambda \sec\Lambda }{\sqrt{3}}E_1+2
   \csc \Lambda\ E_3. \nonumber
\end{align}
We can easily check that 
\begin{align}\label{lieX}
&[X_1,X_2]=2 X_3,\nonumber\\
&[X_2,X_3]=2 X_1,\\
&[X_3,X_1]=2 X_2.\nonumber
\end{align}
Taking a canonical metric on M such that $X_1$, $X_2$ and $X_3$ have unit length and are mutually orthogonal, it follows from the Koszul formula that all connection components are determined. 
From (4.1),  Proposition 5.2 and its preceeding paragraph in \cite{dioosvranckenwang} it follows that we can locally identify M with  $\mathbb{S}^3$ and we can consider $X_1$, $X_2$ and $X_3$ as the standard vector fields on  $\mathbb{S}^3$ with 
\begin{align}
\label{gi}
&X_1(x)=xi,\nonumber\\
&X_2(x)=xj,\\
&X_3(x)=xk.\nonumber
\end{align}
Using the above formulas, the component $p$ of the map can now be determined explicitly.
 First, we write 
\begin{equation}\label{not}
D_{X_i} p=p\alpha_i,
\end{equation} for $i=1,2,3$, where $D$ denotes the Euclidean covariant derivative.
Of course, by Theorem \ref{th1}, $D_{X_1} p=0$. Then, we may compute by \eqref{dp}
\begin{align*}
&(dp(X_2),0)= (\frac{2\cos\Lambda}{\sqrt{3}}+2\sin\Lambda) E_2+(-2\cos\Lambda+\frac{\sin\Lambda}{\sqrt{3}})JE_2,\\
&(dp(X_3),0)=(-\frac{2\cos\Lambda}{\sqrt{3}}+2\sin\Lambda) E_3 +(2\cos\Lambda+\frac{\sin\Lambda}{\sqrt{3}})JE_3
\end{align*}
and  we see that 
\begin{equation}
\begin{array}{ll}
 \nabla^E_{X_1}( dp(X_2),0)= (2 dp(X_3),0),& \nabla^E_{X_2}(dp(X_3),0)=(0,0), \\
 \nabla^E_{X_1}(dp(X_3),0)= (-2 dp(X_2),0),&  \nabla^E_{X_3}( dp(X_2),0)= (0,0),  \\
 \nabla^E_{X_2}(dp(X_2),0)= (0,0),&  \nabla^E_{X_3} (dp(X_3),0)= (0,0).
\end{array}
\end{equation}
Moreover, it is straightforward to get
\begin{equation}\label{lungime}
  \langle dp(X_2),dp(X_2) \rangle=\langle dp(X_3),dp(X_3) \rangle=4, \quad  \langle dp(X_2),dp(X_3) \rangle=0.
\end{equation}
Next, we want to determine a system of differential equations satisfied by $\alpha_2$ and $\alpha_3$.
For this, we consider $\nks\in \mathbb{R}^4\times\mathbb{R}^4$. On the one hand,  we use  \eqref{not}  together with $D_X(dp(Y),0)=(D_Xdp(Y),0)$. On the other hand,  we use \eqref{labe}   and, therefore,  we  obtain
\begin{equation}
\begin{array}{ll}
\label{syst}
X_1(\alpha_2)=2\alpha_3,& X_1(\alpha_3)=-2\alpha_2,\\
X_2(\alpha_2)=0,&X_2(\alpha_3)=-\alpha_2 \times\alpha_3,\\
X_3(\alpha_2)=-\alpha_3\times\alpha_2, & X_3(\alpha_3)=0.
\end{array}
\end{equation}
We choose a unit quaternion $h$  such that at the point $p(x)=1$ we have
\begin{align*}
&\alpha_2(1)=-2hjh^{-1},\\
&\alpha_3(1)=-2hkh^{-1},\\
&\alpha_2\times\alpha_3(1)=4hih^{-1}.
\end{align*}
Using \eqref{gi}, we can check that $\alpha_2=-2hxjx^{-1}h^{-1}$, $\alpha_3=-2hxkx^{-1}h^{-1}$ and $\alpha_2\times\alpha_3=4hxix^{-1}h^{-1}$
are the unique solutions for the  system of differential equations in \eqref{syst}:
\begin{align*}
X_1(\alpha_2)=X_1(-2hxjx^{-1}h^{-1})&=-2(hX_1(x)jx^{-1}h^{-1}+hxjX_1(x^{-1})h^{-1})\\
&=-4hxkx^{-1}h^{-1}\\
&=2\alpha_3,
\end{align*}
\begin{align*}
X_1(\alpha_3)=X_1(-2hxkx^{-1}h^{-1})&=-2(hX_1(x)kx^{-1}h^{-1}+hxkX_1(g^{-1})h^{-1})\\
&=4hxjx^{-1}h^{-1}\\
&=-2\alpha_2,
\end{align*}
\begin{align*}
X_2(\alpha_3)=X_2(-2hxkx^{-1}h^{-1})&=-2(hxjkx^{-1}h^{-1}+hxk(-j)x^{-1}h^{-1})\\
&=-4hxix^{-1}h^{-1}\\
&=-\alpha_2\times\alpha_3,
\end{align*}
\begin{align*}
X_2(\alpha_2)=X_2(-2hxjx^{-1}h^{-1})&=-2(hxjjx^{-1}h^{-1}+hxj(-j)x^{-1}h^{-1})\\
&=0,
\end{align*}
\begin{align*}
X_3(\alpha_3)=X_3(-2hxkx^{-1}h^{-1})&=-2(hxkk x^{-1}h^{-1}+hxk(-k)x^{-1}h^{-1})\\
&=0,
\end{align*}
\begin{align*}
X_3(\alpha_2)=X_3(-2hxjx^{-1}h^{-1})&=-2(hxkj x^{-1}h^{-1}+hxj(-k)x^{-1}h^{-1})\\
&=4hxix^{-1}h^{-1}\\
&=\alpha_2\times\alpha_3.
\end{align*}
This in its turn implies that
\begin{equation}
\label{p}
p(x)=-hixix^{-1}h^{-1}
\end{equation} 
is the unique solution of $X_i(p)= p \alpha_i$ with initial conditions $p(1)=1$. Indeed we have
\begin{align*}
&X_1(p)=X_1(-hixix^{-1}h^{-1})=0=p \alpha_1,\\
&X_2(p)=X_2(-hixix^{-1}h^{-1}) = 2 hixk x^{-1}h^{-1}=(-hi xix^{-1}h^{-1})(-2 hxjx^{-1}h^{-1})=p\alpha_2,\\
&X_3(p)=X_3(-hixix^{-1}h^{-1}) = -2 hixj x^{-1}h^{-1}=(-hi xix^{-1}h^{-1})(-2 hxkx^{-1}h^{-1})=p\alpha_3.
\end{align*}
Before we can determine the second component $q$ of the Lagrangian immersion, we need to explore the Codazzi equations further. First we look at the system of differential equations for the function $\Lambda$ in \eqref{rel1} and \eqref{rel1prim}. Notice that by using the relations in \eqref{newbasis} we have that
\begin{equation}\label{newstar}
\begin{array}{l}
X_1(\Lambda) = 0,\\
X_2(\Lambda) = 2 h_{23}^3 \csc \Lambda ,\\
 X_3(\Lambda) = -2 h_{22}^3 \csc \Lambda ,
\end{array}
\end{equation}
where the last two equations can be seen as the definition for the functions 
$h_{23}^3$ and $h_{22}^3$. The first one is, of course, a condition for the unknown function of $\Lambda$. Three out of the four Codazzi equations then can be seen as integrability conditions for the existence of a solution of this system, whereas the last one 
reduces to 
$$X_2(X_2(\Lambda))+X_3(X_3(\Lambda)) = (\cot(\Lambda) -\tan(\Lambda))( (X_2(\Lambda))^2+(X_3(\Lambda))^2) +4(1+2 \cos(2 \Lambda)) \cot(\Lambda).$$
Under the change of variable $\Lambda=\arctan(e^{2\beta})$, this  equation simplifies to 
\begin{equation}\label{ec}
X_2(X_2(\beta))+X_3(X_3(\beta))=\frac{2(3-e^{4\beta})}{e^{4\beta}}.
\end{equation}
Note also that for $\Lambda = \pm \frac{2 \pi}{3}$, we get the trivial solutions of the problem (corresponding to the constant sectional curvature case).
\begin{remark} Note that there exist at least locally many solutions of the system
\begin{align*}
&X_1(\beta)= 0,\\
&X_2(X_2(\beta))+X_3(X_3(\beta))=\frac{2(3-e^{4\beta})}{e^{4\beta}}.
\end{align*}
This can be seen by choosing special coordinates on the usual  $\mathbb{S}^3$. 
We take
\begin{align*}
x_1=\cos v\cos(t+u),\\
x_2=\cos v\sin(t+u),\\
x_3=\sin v\cos(u-t),\\
x_4=\sin v\sin(u-t).\\
\end{align*} 
As, given \eqref{gi}, at the point $x=(x_1,x_2,x_3,x_4)$ the vectors in the basis are 
\begin{align*}
X_1(x)=(-x_2,x_1,x_4,-x_3),\\
X_2(x)=(-x_3,-x_4,x_1,x_2),\\
X_3(x)=(-x_4,x_3,-x_2,x_1),
\end{align*}
 it is straightforward to see that
\begin{align*}
&\partial t=X_1,\\
&\partial u=\cos(2v)X_1+\sin(2t)\sin(2v)X_2+\cos(2t)\sin(2v)X_3,\\
&\partial v=\cos(2t)X_2-\sin(2t)X_3,
\end{align*}
and conversely,
\begin{align*}
&X_1=\partial u,\\
&X_2=\frac{\sin(2t)}{\sin(2v)}\partial u-\sin(2t)\frac{\cos(2v)}{\sin(2v)}\partial t+\cos(2t)\partial v,\\
&X_3=\frac{\cos(2t)}{\sin(2v)}\partial u-\cos(2t)\frac{\cos(2v)}{\sin(2v)}\partial t-\sin(2t)\partial v.
\end{align*}
At last, the equations in \eqref{ec} become $\tfrac{\partial}{\partial t} \beta=0$ and 
\begin{equation}
\csc^2(2v)\frac{\partial^2\beta}{\partial u^2}+\frac{\partial^2\beta}{\partial v^2}+2\cot(2v)\frac{\partial \beta}{\partial v}=2(3e^{-4\beta}-1).
\end{equation}
\end{remark}
In the following part we are going to determine the second part of the immersion. We start with an arbitrary solution of 
\begin{align*}
&X_1(\beta)= 0,\\
&X_2(X_2(\beta))+X_3(X_3(\beta))=\frac{2(3-e^{4\beta})}{e^{4\beta}}
\end{align*}
and we are going to find a system of differential equations determining the immersion $q$. We define $h_{22}^3$ and $h_{23}^3$ as in \eqref{newstar} and such that $\lambda=\arctan(e^{2\beta})$.  First, we  can write for each of the bases that we took, $\{E_i\}$ and $\{X_i\}$, the following:
\[ 
 \begin{array}{ccc}  \label{system}
X_1(q)=q\beta_1, & X_1(p)=p\alpha_1, \\
X_2(q)=q\beta_2, & X_2(p)=p\alpha_2, \\
X_3(q)=q\beta_3, & X_3(p)=p\alpha_3,
 \end{array}  
\text{and}
  \begin{array}{ccc}
E_1(q)=q\tilde{\beta}_1, & E_1(p)=p\tilde\alpha_1, \\
E_2(q)=q\tilde{\beta}_2, & E_2(p)=p\tilde\alpha_2, \\
E_3(q)=q \tilde{\beta}_3, & E_3(p)=p \tilde\alpha_3,
 \end{array} \]
where $\alpha_1=0$ and $\alpha_2$ and $\alpha_3$ are as determined previously.
 Then,  we  prove as before that
\begin{align}
\label{alphabeta}
&\tilde\beta_1=-\frac{\sqrt{3}}{2\sin^2\Lambda}\tilde\alpha_2\times \tilde\alpha_3,\nonumber\\
&\tilde\beta_2=\frac{\cos(2\Lambda+\frac{2\pi}{3})-\frac{1}{\sqrt{3}} \sin(2\Lambda+\frac{2\pi}{3})}{1-\frac{2}{\sqrt{3}}\sin(2\Lambda+\frac{2\pi}{3})}\tilde\alpha_2=\tfrac 12 (1-\sqrt{3} \cot \Lambda) \tilde \alpha_2,\\
&\tilde\beta_3=\frac{\cos(-2\Lambda+\frac{2\pi}{3})-\frac{1}{\sqrt{3}} \sin(-2\Lambda+\frac{2\pi}{3})}{1-\frac{2}{\sqrt{3}}\sin(-2\Lambda+\frac{2\pi}{3})}\tilde\alpha_3=\tfrac 12 (1+\sqrt{3} \cot \Lambda) \tilde \alpha_3\nonumber
\end{align}
and  we continue the computations in order to find the system of differential equations for the immersion $q$ in terms of the basis $\{X_i\}$.
As we identify $df(X_1)\equiv X_1$, we have
$$
D_{X_1}f=(X_1(p),X_1(q))=(0,q\beta_1)\equiv X_1\overset{\eqref{newbasis}}{=}\frac{4}{\sqrt{3}}E_1=\frac{4}{\sqrt{3}}(p\tilde\alpha_1,q\tilde\beta_1).
$$
Therefore, $\beta_1=\frac{4}{\sqrt{3}}\tilde\beta_1.$ We may compute similarly for $D_{X_2}f$ and $D_{X_3}f$ and find
\[ \begin{array}{cc}
\left\{
\begin{array}{l}
\beta_2=2 \csc \Lambda\ \tilde\beta_2-\frac{2}{\sqrt{3}} h_{22}^3 \csc ^2\Lambda \sec \Lambda \tilde\beta_1,\\
\beta_3=2 \csc \Lambda\ \tilde\beta_3-\frac{2}{\sqrt{3}} h_{23}^3 \csc ^2\Lambda \sec \Lambda \tilde\beta_1,
\end{array}
\right.
& 
\left\{
\begin{array}{l}
\tilde\alpha_2=-\frac{1}{\csc\Lambda}hxjx^{-1}h^{-1},\\
\tilde\alpha_3=-\frac{1}{\csc\Lambda}hxkx^{-1}h^{-1}\\
\end{array}
\right.
 \end{array} \]
 and
 $$\tilde \beta_1=-\frac{\sqrt{3}}{2}hxix^{-1}h^{-1}.$$ 
Using now relations \eqref{alphabeta} we may express
\begin{align*}
\beta_2=-(1-\sqrt{3} \cot\Lambda)\  hxjx^{-1}h^{-1}+h_{22}^3 \csc ^2\Lambda \sec \Lambda \ hxix^{-1}h^{-1},\\
\beta_3=-(1+\sqrt{3} \cot\Lambda) \  hxkx^{-1}h^{-1} + h_{23}^3 \csc ^2\Lambda \sec \Lambda  \ hxix^{-1}h^{-1}.
\end{align*}
Finally, as $X_i(q)=q\beta_i$, we find
\begin{align*}
\left\{
\begin{array}{l}
X_1(q)=-2qhxix^{-1}h^{-1},\\
X_2(q)=q(h_{22}^3 \csc ^2\Lambda \sec \Lambda \ hxix^{-1}h^{-1}-(1-\sqrt{3} \cot\Lambda)\ hxjx^{-1}h^{-1}),\\
X_3(q)=q( h_{23}^3 \csc ^2\Lambda \sec \Lambda  \ hxix^{-1}h^{-1} -(1+\sqrt{3} \cot\Lambda) \ hxkx^{-1}h^{-1}),
\end{array}
\right.
\end{align*}
which, given \eqref{newstar} and $\lambda=\arctan(e^{2\beta})$, is equivalent to 
\begin{align*}
\left\{
\begin{array}{l}
X_1(q)=-2qhxix^{-1}h^{-1},\\
X_2(q)=q\left(-X_3(\beta)  hxix^{-1}h^{-1}-(1-\sqrt{3} e^{-2\beta})\ hxjx^{-1}h^{-1}\right),\\
X_3(q)=q\left( X_2(\beta)\ hxix^{-1}h^{-1} -(1+\sqrt{3} e^{-2\beta}) \ hxkx^{-1}h^{-1}\right).
\end{array}
\right.
\end{align*}
By straightforward computations, one may see that $ X_i(X_j(q))-X_j(X_i(q))=[X_i,X_j](q)$ hold for $i,j=1,2,3$.

\subsubsection{\textbf{Case 3. The minimal surface $p(M)$ is not totally geodesic, but the map $\mathcal P$ is not an immersion}}

As mentioned before this means that
\begin{equation}
 h_{12}^3=-\frac{\sin(2 \Lambda)}{\sqrt{3}}.
\end{equation}
Therefore, the equations in subsection \ref{subs} which follow from \eqref{opA} and \eqref{opB} become
\[ \begin{array}{lll}
h_{12}^2=-h_{13}^3,& \omega_{12}^3=\frac{1+2\cos( 2\Lambda)}{2\sqrt{3}}, & \omega_{21}^2=\omega_{31}^3=-h_{13}^3\cot \Lambda,\\
h_{11}^2=h_{11}^3=0,& \omega_{21}^3=\frac{1+2\cos (2\Lambda)}{2\sqrt{3}},  &\omega_{31}^2= -\frac{1+2\cos( 2\Lambda)}{2\sqrt{3}}, \\
\omega_{11}^2=\omega_{11}^3=0,& \omega_{22}^3=-h_{22}^3\cot(2\Lambda), &\omega_{32}^3=-h_{23}^3\cot( 2\Lambda)
 \end{array} \]
 and 
\begin{equation}\label{derivlambda}
E_1(\Lambda)=h_{13}^3,\quad E_2(\Lambda)=h_{23}^3,\quad E_3(\Lambda)=-h_{22}^3.
\end{equation}
Moreover, the equations of Codazzi in \eqref{cdz0} yield $h_{13}^3=0$ and, therefore, $\omega_{21}^2=\omega_{31}^3=0$. The first two equations in \eqref{cdz} imply that 
$$
E_1(h_{23}^3)=0\quad \text{and}\quad E_1(h_{22}^3)=0,
$$
while the next three ones vanish identically.
The last two equations in \eqref{cdz} become 
\begin{equation}
 \label{6}
E_2(h_{22}^3)=-E_3(h_{23}^3)
\end{equation}
and 
\begin{equation}\label{ec7}
-1-[1+6(h_{22}^3)^2+6(h_{23}^3)^2]\cos 2\Lambda+\cos 4\Lambda+\cos 6\Lambda+2[-E_3(h_{22}^3)+E_2(h_{23}^3)]\sin 2\Lambda=0,
\end{equation}
respectively. The Lie brackets of the vector fields $E_1,E_2,E_3$ give
\begin{align*}
&[E_1,E_2]=0,\\
&[E_1,E_3]=0,\\
&[E_2,E_3]=-\frac{1+2\cos (2\Lambda)}{\sqrt{3}}E_1+h_{22}^3\cot (2\Lambda) E_2+h_{23}^3\cot( 2\Lambda ) E_3.
\end{align*}
Next, we take new vector fields $X_1,X_2,X_3$ of the form 
\begin{align}
&X_1=E_1,\nonumber\\
&X_2=\frac{\sqrt{2}(h_{22}^3-h_{23}^3)}{3^{\frac{3}{4}}(\sin (2\Lambda))^\frac{3}{2}}E_1+\frac{\sqrt{2}}{3^{\frac{1}{4}}\sqrt{\sin (2\Lambda)}}E_2-\frac{\sqrt{2}}{3^{\frac{1}{4}}\sqrt{\sin (2\Lambda)}}E_3,\label{vectors1}\\
&X_3=\frac{\sqrt{2}(h_{22}^3+h_{23}^3)}{3^{\frac{3}{4}}(\sin (2\Lambda))^\frac{3}{2}}E_1+\frac{\sqrt{2}}{3^{\frac{1}{4}}\sqrt{\sin( 2\Lambda)}}E_2+\frac{\sqrt{2}}{3^{\frac{1}{4}}\sqrt{\sin (2 \Lambda)}}E_3.\nonumber
\end{align}
We can easily check that $[X_1,X_2]=0$, $[X_1,X_3]=0$ and $[X_2,X_3]=0$, therefore, by the lemma on page 155 in \cite{boothby}, we know that there exist coordinates $\{t,u,v\}$ on $M$ such that 
\begin{align*}
X_1=\partial t,\\
X_2=\partial u,\\
X_3=\partial v.
\end{align*}
 Using \eqref{derivlambda} we obtain:
\begin{align*}
&\Lambda_t=0,\\
&\Lambda_u=\frac{h_{22}^3+h_{23}^3}{3^{1/4}\sqrt{\cos\Lambda\sin\Lambda}},\\
&\Lambda_v=\frac{-h_{22}^3+h_{23}^3}{3^{1/4}\sqrt{\cos\Lambda\sin\Lambda}}.
\end{align*}
Furthermore, we express $h_{22}^3$ and $h_{23}^3$ from the previous relations as
\begin{align*}
h_{22}^3=\frac{1}{2}3^{1/4}(\Lambda_u-\Lambda_v)\sqrt{\cos\Lambda\sin\Lambda},\\
h_{23}^3=\frac{1}{2}3^{1/4}(\Lambda_u+\Lambda_v)\sqrt{\cos\Lambda\sin\Lambda}
\end{align*}
and therefore, the expression of \eqref{vectors1} becomes
\begin{align}
&X_1=E_1,\nonumber\\
&X_2=-\frac{\Lambda_v \csc (2\Lambda)}{\sqrt{3}}E_1+\frac{1}{3^{\frac{1}{4}}\sqrt{\cos\Lambda\sin\lambda}} E_2-\frac{1}{3^{\frac{1}{4}}\sqrt{\cos\Lambda\sin\lambda}} E_3\label{X},\\
&X_3=\frac{\Lambda_u \csc (2\Lambda)}{\sqrt{3}}E_1+\frac{1}{3^{\frac{1}{4}}\sqrt{\cos\Lambda\sin\lambda}} E_2+\frac{1}{3^{\frac{1}{4}}\sqrt{\cos\Lambda\sin\lambda}} E_3.\nonumber
\end{align}
Finally, by  straightforward computations, one may see that equation  \eqref{ec7} becomes
\begin{multline}\label{eqq7}
-\sqrt{3}(\Lambda_u^2+\Lambda_v^2)\cos (2\Lambda)+3^{\frac{1}{4}}(E_2(\Lambda_u)-E_3(\Lambda_u)+E_2(\Lambda_v)+E_3(\Lambda_v))\sqrt{\cos\Lambda\sin\Lambda}-\\-2(\sin(2\Lambda)+\sin (4\Lambda))=0.
\end{multline}
We compute $dp(\partial u)$ and $dp(\partial v)$: 
\begin{multline}
dp\left(\partial u\right)=\frac{\sqrt{3}-2 \cos \left(2\Lambda +\frac{\pi }{6}\right)}{ 3^{3/4} \sqrt{2} \sqrt{\sin( 2\Lambda)}}E_2+\frac{2 \sin
   \left( 2\Lambda +\frac{\pi}{3}\right)-\sqrt{3}}{3^{3/4}\sqrt{2}  \sqrt{\sin (2\Lambda) }}E_3+\frac{2 \cos \left( 2\Lambda +\frac{2\pi}{3}
   \right)+1}{3^{3/4} \sqrt{2} \sqrt{\sin(2\Lambda) }}JE_2+\\
+\frac{2 \cos \left( 2\Lambda +\frac{\pi}{3} \right)-1}{{ 3^{3/4}\sqrt2}  \sqrt{\sin (2\Lambda)}}JE_3,\nonumber
\end{multline}
\begin{multline}
dp(\partial v)=\frac{\sqrt{3}-2 \cos \left(2\Lambda +\frac{\pi }{6}\right)}{ 3^{3/4} \sqrt{2} \sqrt{\sin( 2\Lambda)}}E_2-\frac{2 \sin
   \left( 2\Lambda +\frac{\pi}{3}\right)-\sqrt{3}}{3^{3/4}\sqrt{2}  \sqrt{\sin (2\Lambda) }} E_3+\frac{2 \cos \left( 2\Lambda +\frac{2\pi}{3}
   \right)+1}{3^{3/4} \sqrt{2} \sqrt{\sin(2\Lambda) }} JE_2-\\-\frac{2 \cos \left( 2\Lambda +\frac{\pi}{3} \right)-1}{{ 3^{3/4}\sqrt2}  \sqrt{\sin (2\Lambda)}}JE_3, \nonumber
\end{multline}
and we remark that they are mutually orthogonal and that their length is
$\frac{2\tan \Lambda}{\sqrt{3}}.$ So, as $u,v$ are isothermal coordinates on the surface, for which $\langle \partial u,\partial u \rangle=\langle \partial v,\partial v \rangle=2 e^\omega $, we obtain that 
\begin{equation}\label{egalitate}
e^\omega=\frac{\tan \Lambda}{\sqrt{3}}.
\end{equation}
 On the one hand, for  $z=x+Iy$ as in subsection \ref{minimalsurface}, we may compute $dp(\partial z)$:
\begin{align*}
dp(\partial z)&=\frac{1}{2}\left[ dp\left(\partial u\right)-I\cdot dp\left(\partial v\right)\right]\\
&=\frac{1}{2\sqrt{2}\ 3^{3/4}\sqrt{\sin(2\Lambda)}}\left[ (1-I)\left( \sqrt{3}-2\cos\left(2\Lambda+\frac{\pi}{6}\right)  \right)E_2- \right. \\
&- (1+I)\left( \sqrt{3}-2\sin \left(2\Lambda+\frac{\pi}{3}\right)  \right)E_3 + (1-I)\left( 2\cos \left(2\Lambda+\frac{2\pi}{3}\right)+1  \right)JE_2+ \\
&\left.+  (1+I)\left( 2\cos \left(2\Lambda+\frac{\pi}{3}\right)-1  \right)JE_3\right].
\end{align*}
As 
\begin{align*}
 \sqrt{3}-2\cos\left(2\Lambda+\frac{\pi}{6}\right)=2\sin\Lambda(\sqrt{3}\sin\Lambda+\cos\Lambda),\quad 
2\sin \left(2\Lambda+\frac{\pi}{3}\right) - \sqrt{3}=2\sin\Lambda(\cos\Lambda-\sqrt{3}\sin\Lambda),\\
2\cos \left(2\Lambda+\frac{2\pi}{3}\right)+1 =2\sin\Lambda(\sin\Lambda-\sqrt{3}\cos\Lambda),\quad
2\cos \left(2\Lambda+\frac{\pi}{3}\right)-1=-2\sin\Lambda(\sin\Lambda+\sqrt{3}\cos\Lambda),
\end{align*}
we finally have
\begin{align*}
dp(\partial z)&=\frac{\sin\Lambda}{\sqrt{2}\ 3^{3/4}\sqrt{\sin(2\Lambda)}}\left[ (1-I)\left( \sqrt{3}\sin\Lambda+\cos\Lambda  \right)E_2+ (1+I)\left(\cos\Lambda - \sqrt{3}\sin\Lambda \right)E_3 +\right. \\
&  \left.  +  (1-I)\left(\sin\Lambda - \sqrt{3}\cos\Lambda \right)JE_2- (1+I)\left(\sin\Lambda + \sqrt{3}\cos\Lambda \right)JE_3    \right].
\end{align*}
Moreover, from \eqref{egalitate}, it follows that $\omega_z=\frac{1}{\sin(2\Lambda)}\left(\Lambda_u-i\Lambda_v\right)$.\\
On the other hand, we may compute $\nabla^E_{\partial z}dp(\partial z)$ using the Euclidean connection $\nabla^E$ :
\begin{align*}
\nabla^E_{\partial z} dp(\partial z)&=-\frac{1}{\sqrt{3}}E_1+\frac{e^{-\frac{i \pi }{4}} \sin ^2\Lambda \left(\sqrt{3} \cot \Lambda +3\right) (\Lambda_u-i \Lambda_v)}{3
   \sqrt[4]{3} \sin ^{\frac{3}{2}}(2\Lambda )}E_2+\\
&+\frac{e^{-\frac{i \pi }{4}} \sin\Lambda (\Lambda_v+i\Lambda_u) \left(\sqrt{3} \cos
   \Lambda-3 \sin \Lambda \right)}{3 \sqrt[4]{3} \sin ^{\frac{3}{2}}(2 \Lambda)}E_3+JE_1+\\
&+\frac{e^{-\frac{i \pi }{4}} \sin\Lambda( \Lambda_u-i
   \Lambda_v) \left(\sqrt{3} \sin \Lambda-3 \cos\Lambda\right)}{3 \sqrt[4]{3} \sin ^{\frac{3}{2}}(2 \Lambda
   )}JE_2-\\
&-\frac{\left(\frac{1}{3}+\frac{i}{3}\right) \sin\Lambda (\Lambda_u-i \Lambda_v) \left(\sqrt{3} \sin \Lambda+3 \cos\Lambda\right)}{\sqrt{2} \sqrt[4]{3} \sin ^{\frac{3}{2}}(2 \Lambda )}JE_3.
\end{align*}
From the previous computations we see, indeed, that 
\begin{align*}
\nabla^E_{\partial z}dp(\partial z)&=-N+\omega_z dp(\partial z),
\end{align*}
which corresponds to \eqref{secondder}.
 From here, we remark the component in the direction of the normal $N=\xi$ (see subsection \ref{subs}) and we see that  the choice of coordinates $\{t,u,v\}$ following from \eqref{vectors1} is the right one, as we have indeed  $\sigma(\partial z,\partial z)=-1$, as in subsection \ref{minimalsurface}.
Using \eqref{egalitate} together with the fact that, by taking the inverse in \eqref{X}, we have
\begin{align*}
&E_1=\partial t,\\
&E_2=\frac{3^{\frac{1}{4}} \sqrt{\sin (2\Lambda)}}{2\sqrt{2}} \left( \frac{\Lambda_v-\Lambda_u}{\sqrt{3}\sin( 2\Lambda) }\partial t+\partial u+\partial v \right),\\
&E_3=-\frac{3^{\frac{1}{4}} \sqrt{\sin (2\Lambda)}}{2\sqrt{2}} \left( \frac{\Lambda_v+\Lambda_u}{\sqrt{3}\sin( 2\Lambda) }\partial t+\partial u-\partial v \right),\\
\end{align*}
we may prove that equation \eqref{eqq7}  is equivalent to the Sinh-Gordon equation in \eqref{lap}, which  carachterizes the minimal surface.

\hfill\break
\textbf{ Reverse construction} \hfill\break

Let $S$ be a minimal surface given by $p:S\rightarrow \mathbb{S}^3\subset\mathbb{R}^4$, on which we take isothermal coordinates $ u$ and $v$ as in subsection \ref{minimalsurface}.  Hence, we have a solution $\omega$ of the Sinh-Gordon equation $\Delta\omega=-8\sinh\omega.$
Next, we define a function $\Lambda\in[0,\frac{\pi}{2}]$ such that 
$$
e^\omega=\frac{\tan\Lambda}{\sqrt{3}}.
$$
\begin{remark}
If $\omega=0$, then $\Lambda=\frac{\pi}{3}$, which corresponds to example (2) in Theorem \ref{th1}.
\end{remark}
We then define a metric on an open part of the unit frame bundle of the surface by assuming that the vectors 
\begin{align}
&E_1=\partial t,\nonumber\\
&E_2=\frac{\sqrt{3}e^{\omega/2}}{2\sqrt{1+3 e^{2\omega}}}\left(\frac{\omega_v-\omega_u}{2\sqrt{3}}\partial t+\partial u+\partial v \right ),\label{vectors}\\
&E_3=-\frac{\sqrt{3}e^{\omega/2}}{2\sqrt{1+3 e^{2\omega}}}\left(\frac{\omega_v+\omega_u}{2\sqrt{3}}\partial t+\partial u-\partial v \right )\nonumber
\end{align}
form an orthonormal basis.
Next, we want to determine the Lagrangian immersion 
\begin{align*}
&f:S\times I\rightarrow\mathbb{S}^3\times\mathbb{S}^3\\
&(u,v,t)\mapsto f(u,v,t)=(p(u,v,t),q(u,v,t)),
\end{align*}
for which we already know that the first component is the given minimal surface $p$. We write for both bases
\[ 
 \begin{array}{ccc}  \label{system}
 \frac{\partial}{\partial t}(q)=q\beta_1, &  \frac{\partial}{\partial t}(p)=p\alpha_1, \\
 \frac{\partial}{\partial u}(q)=q\beta_2, &  \frac{\partial}{\partial u}(p)=p\alpha_2, \\
 \frac{\partial}{\partial v}(q)=q\beta_3, &  \frac{\partial}{\partial v}(p)=p\alpha_3
 \end{array}  
\text{and}
  \begin{array}{ccc}
E_1(q)=q\tilde{\beta}_1, & E_1(p)=p\tilde\alpha_1, \\
E_2(q)=q\tilde{\beta}_2, & E_2(p)=p\tilde\alpha_2, \\ 
E_3(q)=q \tilde{\beta}_3, & E_3(p)=p \tilde\alpha_3.
 \end{array} \]
Note that $\alpha_1=0$ and $\alpha_2$ and $\alpha_3$ are determined by the minimal surface. In particular $\alpha_2$ and $\alpha_3$ are mutually orthogonal imaginary quaternions with length squared $2 e^\omega$.
 From the derivates of $p$ in the latter relations together with \eqref{vectors}, we obtain
\begin{align}
&\tilde\alpha_1=0,\nonumber\\
&\tilde\alpha_2=\frac{\sqrt{3}e^{\omega/2}}{2\sqrt{1+3e^{2\omega}}}(\alpha_2+\alpha_3),\label{tildealphas}\\
&\tilde\alpha_2=-\frac{\sqrt{3}e^{\omega/2}}{2\sqrt{1+3e^{2\omega}}}(\alpha_2-\alpha_3).\nonumber
\end{align}
We then follow the same steps as in \emph{Case 1} and obtain 
\begin{align}
&\tilde\beta_1=-\frac{\sqrt{3}e^{-\omega}}{4}\alpha_2\times\alpha_3,\nonumber\\
&\tilde\beta_2=\frac{\sqrt{3}(e^{\omega/2}-e^{-\omega/2})}{4\sqrt{1+3e^{2\omega}}}(\alpha_2+\alpha_3),\\
&\tilde\beta_3=-\frac{\sqrt{3}(e^{\omega/2}-e^{-\omega/2})}{4\sqrt{1+3e^{2\omega}}}(\alpha_2-\alpha_3).\nonumber
\end{align}
Finally, we take the inverse of the matrix which give  $\{E_i\}$ in the basis $\{ \partial t, \partial u,\partial v\}$ in \eqref{vectors}  and obtain 
\begin{align*}
&\beta_1=-\frac{\sqrt{3}e^{-\omega}}{4}\alpha_2\times\alpha_3,\\
&\beta_2=\frac{e^{-\omega}}{8}(4e^{\omega}\alpha_2-4\alpha_3+\omega_v\alpha_2\times\alpha_3),\\
&\beta_3=-\frac{e^{-\omega}}{8}(4\alpha_2-4e^{\omega}\alpha_3+\omega_u\alpha_2\times\alpha_3).
\end{align*}
By straightforward computations, it now follows that
\begin{align*}
&\frac{\partial \beta_1}{\partial u}-\frac{\partial \beta_2}{\partial t}-2 \beta_1 \times \beta_2 =0,\\
&\frac{\partial \beta_1}{\partial v}-\frac{\partial \beta_3}{\partial t}-2 \beta_1 \times \beta_3 =0,\\
&\frac{\partial \beta_3}{\partial u}-\frac{\partial \beta_2}{\partial v}-2 \beta_3 \times \beta_2 =0,
\end{align*}
from which we deduce that the integrability conditions for the immersion $q$ are satsified.

\section{Conclusion}\label{conclusion}

The results in Section \ref{section3.3} can now be summarized in the following theorems.

\begin{theorem}\label{t1}
Let $\omega$ and $\mu$ be solutions of, respectively, the Sinh-Gordon equation $\Delta\omega=-8\sinh\omega$ and the Liouville equation  $\Delta\mu=-e^\mu$ on an open simply connected domain $U \subseteq\mathbb{C}$ and let $p:U\rightarrow\mathbb{S}^3$ be the associated minimal surface with complex coordinate $z$ such that $\sigma(\partial z,\partial z)=-1.$\\
Let $V=\{(z,t)\mid z\in U, t\in \mathbb{R}, e^{\omega+\mu}-2-2\cos(4t)>0 \}$ and let $\Lambda$ be a solution of $$\left(\frac{2\sqrt{3}e^\omega}{\tan\Lambda}-2\sin(2t) \right)= e^{\omega+\mu}-2-2\cos(4t)$$ on $V$. Then, there exists a Lagrangian immersion $f:V\rightarrow \nks : x\mapsto (p(x),q(x))$, where $q$ is determined by 
\begin{align*}
\frac{\partial q}{\partial t}=& -\frac{\sqrt{3} }{2 \sqrt{3} e^{\omega }-2 \sin
   (2t) \tan \Lambda }\ q\ \alpha_2 \times \alpha_3, \\
\frac{\partial q}{\partial u}=& \frac{1}{8} \Bigl(e^{-\omega } \left( {\mu_v}+ {\omega_v}-\frac{( {\mu_u}+ {\omega_u}) \cos
   (2 t) \tan \Lambda }{\sqrt{3} e^{\omega }
  -\sin (2 t) \tan \Lambda }\right)q \ \alpha_2 \times \alpha_3  - 4 (\sqrt{3} \cot \Lambda \cos (2 t)+1)\ q\  \alpha_2 - \Bigr. \\ 
& \Bigl. -4 \sqrt{3}\sin (2 t) \cot\Lambda \ q\ \alpha_3   \Bigr), \\
\frac{\partial q}{\partial v}=& \frac{1}{8} \Bigl(-e^{-\omega } \left( {\mu_u}+ {\omega_u}+\frac{( {\mu_v}+ {\omega_v}) \cos
   (2 t) \tan\Lambda }{\sqrt{3} e^{\omega }
  -\sin (2 t) \tan \Lambda }\right) \ q\ \alpha_2 \times \alpha_3  - 4 \sqrt{3} \cot \Lambda \sin (2 t)\  q\ \alpha_2+ \Bigr. \\ &\Bigl.  +4(1+ \sqrt{3}\cos (2 t) \cot\Lambda)\ q \ \alpha_3   \Big),
\end{align*}
where $\alpha_2=\bar p p_u$ and $\alpha_3=\bar p p_v$.
\end{theorem}

\begin{theorem}\label{t2}
Let $X_1,X_2,X_3$ be the standard vector fields on $\mathbb{S}^3$. Let $\beta$ be a solution of the differential equations 
\begin{align*}
&X_1(\beta)=0,\\
&X_2(X_2(\beta))+X_3(X_3(\beta))=\frac{2(3-e^{4\beta})}{e^{4\beta}},
\end{align*}
on a connected, simply connected open subset $U$ of $\mathbb{S}^3$.\\
 Then there exist a Lagrangian immersion $f:U\rightarrow \nks : x\mapsto (p(x),q(x))$, where $p(x)=xix^{-1}$ and $q$ is determined by
\begin{align*}
\begin{array}{l}
X_1(q)=-2qhxix^{-1}h^{-1},\\
X_2(q)=q\left(-X_3(\beta)  hxix^{-1}h^{-1}-(1-\sqrt{3} e^{-2\beta})\ hxjx^{-1}h^{-1}\right),\\
X_3(q)=q\left( X_2(\beta)\ hxix^{-1}h^{-1} -(1+\sqrt{3} e^{-2\beta}) \ hxkx^{-1}h^{-1}\right).
\end{array}
\end{align*}
\end{theorem}

Note that in the previous theorem the image of $p$ is a totally geodesic surface in $\mathbb{S}^3.$

\begin{theorem}\label{t3}
Let $\omega$ be a solution of the Sinh-Gordon equation $\Delta\omega=-8\sinh\omega$ on an open connected domain of $U$ in $\mathbb{C}$ and let $p:U\rightarrow \mathbb{S}^3$ be the associated minimal surface with complex coordinate $z$ such that $\sigma(\partial z,\partial z)=-1.$ Then, there exist a Lagrangian immersion $f:U\times \mathbb{R}\rightarrow\nks:x\mapsto (p(x),q(x)) $, where $q$ is determined by 
\begin{align*}
&\frac{\partial q}{\partial t}=-\frac{\sqrt{3}e^{-\omega}}{4}q \ \alpha_2\times\alpha_3,\\
&\frac{\partial q}{\partial u}=\frac{e^{-\omega}}{8}(4e^{\omega}q \alpha_2-4  q \alpha_3+\omega_v q\ \alpha_2\times\alpha_3),\\
&\frac{\partial q}{\partial v}=-\frac{e^{-\omega}}{8}(4 q \alpha_2-4e^{\omega}q \alpha_3+\omega_u  q\ \alpha_2\times\alpha_3).
\end{align*}
where $\alpha_2=\bar p p_u$ and $\alpha_3=\bar p p_v.$
\end{theorem}

\begin{theorem}
Let $f:M\rightarrow \nks : x\mapsto (p(x),q(x))$ be a Lagrangian immersion such that $p$ has nowhere maximal rank. Then every point $x$ of an open dense subset of $M$ has a neighborhood $U$ such that $f|_U$ is obtained as described in Theorem \ref{t1}, \ref{t2} or \ref{t3}.
\end{theorem}

\end{document}